\documentclass[11pt,draft]{article}
\usepackage{indentfirst}
\usepackage{amsfonts}
\usepackage{amssymb}
\usepackage{mathrsfs}
\usepackage{amsmath}
\usepackage{amsthm}
\usepackage{cite}

\allowdisplaybreaks

\usepackage{geometry}

\usepackage{ifpdf}
\ifpdf
\usepackage[colorlinks=true,linkcolor=blue,citecolor=red, final,hyperindex]{hyperref}
\else
\usepackage[colorlinks,final,hyperindex,hypertex]{hyperref}
\fi

\usepackage{txfonts}
\usepackage{indentfirst}
\usepackage{latexsym}

\usepackage{enumitem}  
\usepackage{calc}

\def\<{\langle}
\def\>{\rangle}

\newtheorem{thm}{Theorem}[section]
\newtheorem{lem}[thm]{Lemma}
\newtheorem{cor}[thm]{Corollary}
\newtheorem{pro}[thm]{Proposition}
\newtheorem{ex}[thm]{Example}

\newtheorem{pdef}[thm]{Proposition-Definition}

\newtheorem{defi}{Definition}[section]

\newtheorem{rmk}{Remark}[section]
\begin{document}
\date{}
\title{Kupershmidt operators on $3$-BiHom-Poisson color algebras and $3$-BiHom-pre-Poisson color algebras structures}

\author{Othmen Ncib$^{1}$\footnote{E-mail: othmenncib@yahoo.fr}\: and Sergei Silvestrov $^{2}$\footnote{E-mail: sergei.silvestrov@mdu.se (Corresponding author)}
}

\date{
$^{1}${\small  University of Gafsa, Faculty of Sciences Gafsa, 2112 Gafsa, Tunisia}\\{\small $^2$ M\"{a}lardalen University,
Division of Mathematics and Physics,}\\
{\small \hspace{1.5 cm}
School of Education, Culture and Communication, }\\
{\small \hspace{1.5 cm}
Box 883, 72123 V\"{a}ster{\aa}s, Sweden.}\\
}
\maketitle
\abstract{
The purpose of this paper is to introduce the class of noncommutative
$3$-BiHom-Poisson color algebras, which is a combination of $3$-BiHom-Lie color algebras and BiHom-associative color algebras under a compatibility condition, called BiHom-Leibniz color identity, and then study their representations and associated Kupershmidt operators. In addition, we introduce the notion of noncommutative
$3$-BiHom-pre-Poisson color algebras and investigate the relationship between noncommutative $3$-BiHom-Poisson color algebras and $3$-BiHom-pre-Poisson color algebras via Kupershmidt operators.}


\newpage 

\tableofcontents

\numberwithin{equation}{section}
\section{
Introduction
}
The notion of Poisson algebra was probably first introduced  by A. M. Vinogradov and J. S. Krasil'shchik in 1975 under the name 'canonical algebra' \cite{VinogradovKrasilshchik75} and by J. Braconnier \cite{Braconnier77} in 1977.
These structures play an important role in many fields in mathematics and mathematical physics, such as Poisson geometry, integrable systems, and non-commutative (algebraic or differential) geometry (see \cite{Vaisman} and the references therein). For example, Poisson algebras play a fundamental role in the deformations of associative commutative algebras \cite{Gerstenhaber}. Moreover, the cohomology, deformations, tensor product, and color generalizations of Poisson algebras have been studied (see, for example, \cite{Ginzburg-Kaledin,Wang-Gao-Zhang,Wu-Zhu-Chen,Zhu-HLi-Li} for more details).

The area of Hom-algebra structures began with the work of Hartwig, Larsson and Silvestrov \cite{Hartwig-Larsson-Silvestrov} in 2003, where general Hom-Lie algebras and more general quasi-Hom-Lie algebras were introduced along with a general method for the construction of quasi-deformations of Witt and Virasoro algebras based on the discretization of vector fields by twisted derivations
($\sigma$-derivations), motivated by the discrete modifications of differential calculi and the $q$-deformed Jacobi identities observed for specific $q$-deformed algebras and their representations in physics and in $q$-deformed differential calculi and homological algebra.
Following \cite{Hartwig-Larsson-Silvestrov}, Larsson and Silvestrov conducted a systematic study of central extensions of quasi-Hom-Lie algebras in \cite{LarssonSilvJA2005:QuasiHomLieCentExt2cocyid} that appeared first in 2004.
At the same time, Larsson and Silvestrov introduced more general quasi-Lie algebras and quasi-Leibniz algebras in \cite{Larsson-Silvestrov-QuasiLiealg} in 2004, and general color quasi-Lie algebras ($Gamma$-graded $\epsilon$-quasi-Lie algebras) in \cite{LSGradedquasiLiealg} in 2005, including general color quasi-Hom-Lie algebras and color Hom-Lie algebras, quasi-Lie superalgebras, quasi-Hom-Lie superalgebras, Hom-Lie superalgebras, and Hom-Lie algebras. Sigurdsson and Silvestrov continued this work in \cite{Sigurdsson-Silvestrov-Czech:witt} in 2006, and in \cite{SigSilvGLTbdSpringer2009-witt-centrext} in 2009, where color quasi-Lie algebras of the Witt type and their central extensions were considered.
Subsequently, various classes of Hom-Lie admissible algebras were considered by Makhlouf and Silvestrov in \cite{ms:homstructure} first appeared in 2006. In particular, in \cite{ms:homstructure}, the Hom-associative algebras have been introduced and shown to be Hom-Lie admissible, that is yielding Hom-Lie algebras using the commutator map as a new product, and in this sense constituting a natural generalization of associative algebras as Lie admissible algebras leading to Lie algebras via the commutator map as a new product.
In \cite{ms:homstructure}, moreover, several other interesting classes of Hom-Lie admissible algebras generalizing some classes of non-associative algebras, as well as examples of finite-dimensional Hom-Lie algebras, have been described.
Hom-algebra structures are very useful since Hom-algebra structures of a given type include their classical counterparts and open more possibilities for deformations, extensions of cohomological structures, and representations.
Since these pioneering works \cite{Hartwig-Larsson-Silvestrov,LarssonSilvJA2005:QuasiHomLieCentExt2cocyid,Larsson-Silvestrov-QuasiLiealg,LSGradedquasiLiealg,LarssonSilvestrov-quasidefal2twderiv,ms:homstructure}, Hom-algebra structures have developed in a popular broad area with an increasing number of publications in various directions, where many authors have studied various classes of 
Hom-algebras (see, for example,
\cite{Ammar-Mabrouk-Makhlouf,Benhassine-Mabrouk-Ncib,Cheng-Su,FSOS,Jin-X-Li,Y-Sheng,Ma-Chen-Lin,Makhlouf-Silvestrov} for more details).

The class of BiHom-algebras was introduced from a categorical approach in \cite{GMMP} as an extension of the class of Hom-algebras. Recall that in a BiHom-algebra the defining identities are twisted by two homomorphisms. When the two homomorphisms are the same, then BiHom algebras become Hom algebras in some cases. Recently, investigations have focused on these classes of algebras in the classic and super cases (see \cite{AAMM,BSO1,CQ,BCMN,CMM1,CMM2,LCC,ncib1} for more details).

The first instance of $n$-ary algebras in physics was in the context of a generalization of Hamiltonian mechanics proposed by Nambu \cite{Nambu:GenHD} in 1973. More recent motivation from string theory and $M$-branes involving naturally an algebra with a ternary operation called Bagger-Lambert algebra stimulated further significant development of $n$-ary algebras in physics. In 1985, Filippov \cite{Filippov} introduced the concept of $n$-Lie algebras and classified the $(n+1)$-dimen\-sional $n$-Lie algebras over an algebraically closed field of characteristic zero. The structure of $n$-Lie algebras is very different from that of Lie algebras because of the $n$-ary multilinear operations. The
$n=3$ case, i.e. $3$-ary multilinear operation, first appeared in
Nambu's work \cite{Nambu:GenHD} in the description of simultaneous classical dynamics of three particles. In that work, Nambu extended the Poisson bracket and arrived at the generalized Hamiltonian equation involving a $3$-ary multilinear bracket $\{\cdot,\cdot,\cdot\}$.
Takhtajan \cite{T} investigated the geometrical and algebraic aspects of generalized Nambu mechanics and established a connection between Nambu mechanics and Filippov's theory of $n$-Lie algebras \cite{Filippov}.

Lie color algebras are natural generalizations of Lie algebras and Lie superalgebras that become an active area of research in mathematics and physics. The cohomology of Lie color algebras was introduced and investigated in \cite{Scheunert1,Scheunert2}, and the representations of Lie color algebras were explicitly described in \cite{Feldvoss}. Several other color algebraic structures were considered by some authors \cite{Armakan-Silvestrov,Bakayoko-Laraiedh,BKP,Bakayoko2,Bakayoko-Silvestrov-multiplnHomLiealgs,Zhang-T}.

Poisson color algebras, considered in \cite{Wang-Gao-Zhang}, have simultaneously an associative color algebra structure and a Lie color algebra structure satisfying the Leibniz identity. Hom-type and BiHom-type generalizations, such as Hom-Poisson color algebras and BiHom-Poisson color algebras, were introduced in \cite{Bakayoko1,Bakayoko2}. The purpose of this paper is to provide a generalization of BiHom-Poisson color algebras in the ternary case. We also define their representations. We also introduce the notion of $3$-BiHom-pre-Poisson color algebras. Finally, we define the notion of Kupershmidt operators in $3$-BiHom-Poisson color algebras and study their relation to $3$-BiHom-pre-Poisson color algebras.
In Section \ref{sec-prelimdefin}, we recall relevant definitions and properties of BiHom-associative color algebras and $3$-BiHom-Lie color algebras and introduce the notions and obtain several results on representations and Kupershmidt operators of these algebraic structures. Section \ref{sec-noncom3bihompoissoncolalgs} discusses the notion of noncommutative $3$-BiHom-Poisson color algebras and their representations, and we define Kupershmidt operators. This section also contains main theorems concerning the representations and Kupershmidt operators of noncommutative $3$-BiHom-Poisson color algebras, as well as the twisted theorem. Section \ref{sec-noncom3bihomprepoisoncolalgs} introduces the notion of noncommutative $3$-BiHom-pre-Poisson color algebras. First, we introduce the notions of $3$-BiHom-pre-Lie color algebras and BiHom-dendriform color algebras, and the reader will obtain some results, among them and the most interesting are those associated with Kupershmidt representations and operators. Finally, by starting from the data of a noncommutative $3$-BiHom-Poisson color algebra $(\mathfrak{g},\{\cdot,\cdot,\cdot\},\epsilon,\alpha,\beta)$ and a Kupershmidt operator $T$ associated with a representation $(V,\rho,\mathfrak l,\mathfrak r,\alpha_V,\beta_V)$, we construct a structure of noncommutative $3$-BiHom-pre-Poisson color algebra on $V$. If $T$ is associated with the adjoint representation called the Rota-Baxter operator, we construct a structure of noncommutative $3$-BiHom-pre-Poisson color algebra on $\mathfrak{g}$. This section also contains some other results.

Throughout this paper, for simplicity of exposition, we assume that $\mathbb{K}$ is an algebraically closed field of characteristic zero, even though for most of the general definitions and results in this paper, this assumption is not essential.

\section{Preliminaries and basic definitions} \label{sec-prelimdefin}
In this section, we recall important basic notions and notations related to graded spaces. We also recall the notions of some algebraic structures, such as noncommutative $3$-BiHom-Lie color algebras and BiHom-Poisson color algebras, and provide examples and related general properties (see \cite{Bakayoko2,Zahra-Esmaeil} for more details). The representation theory and Kupershmidt operators of these algebraic structures are also introduced.
\begin{defi} Let $\Gamma$ be an abelian group.
\begin{enumerate}[label=\upshape{\arabic*.}, ref=\upshape{\arabic*}, labelindent=5pt, leftmargin=*]
\item A linear space $V$ is said to be $\Gamma$-graded if, there exists a family $(V_a)_{a\in \Gamma}$ of linear subspaces of $V$ such that
$V=\bigoplus\limits_{a\in \Gamma} V_a.$
\item
An element $x\in V$ is considered homogeneous of degree $a\in \Gamma$ if $x\in V_a$. We denote $\mathcal{H}(V)$ the set of all homogeneous elements in $V$.
\item  Let $V=\bigoplus\limits_{a\in \Gamma} V_a$ and $V'=\bigoplus\limits_{a\in \Gamma} V'_a$ be two $\Gamma$-graded linear spaces. A Linear mapping $f : V\rightarrow V'$ is said to be homogeneous of degree $b\in \Gamma$ if for all $a\in \Gamma$,
$$f(V_a)\subseteq  V'_{a+b}.$$
If a linear mapping $f$ is homogeneous of degree zero, that is $f(V_a)\subseteq V'_{a}$ for any $a\in \Gamma$, then it is sometimes called an even mapping (as for $\Gamma = \mathbb{Z}_2$ grading case).
\item
An algebra $(A, \cdot)$ is said to be $\Gamma$-graded if its underlying linear space is $\Gamma$-graded, $A=\bigoplus\limits_{a\in \Gamma}A_a$, and if furthermore for all $a, b\in \Gamma$,
\begin{equation} A_a\cdot A_b\subseteq A_{a+b}. \end{equation}
A morphism $f : A\rightarrow A'$
of $\Gamma$-graded algebras $A$ and $A'$
is by definition a morphism of algebras which is even, that is homogeneous mapping of degree zero.
\end{enumerate}
\end{defi}
\begin{defi}
Let $\Gamma$ be an abelian group. A map $\epsilon :\Gamma\times \Gamma\rightarrow {\bf \mathbb{K}^*=\mathbb{K}\setminus \{0\}}$ is called commutation factor (or a skew-symmetric bicharacter) on $\Gamma$ if, for all $a, a', a''\in \Gamma$,
\begin{align}
&\epsilon(a, a')\epsilon(a', a)=1,\label{cond-bic1}\\
& \epsilon(a, a'+a'')=\epsilon(a, a')\epsilon(a, a''),\label{cond-bic2}\\
&\epsilon(a+a', a'')=\epsilon(a, a'')\epsilon(a', a'').\label{cond-bic3}
\end{align}
\end{defi}
In particular, we can observe that $\epsilon(a, 0)=\epsilon(0, a)=1$ and $\epsilon(a,a)=\pm 1$ for all $a\in \Gamma$. If $x$ and $y$ are two homogeneous elements of degree $a$ and $b$ respectively and $\epsilon$ is a skew-symmetric bicharacter, then the notation $\epsilon(x, y)$ is often used instead of $\epsilon(a, b)$.
\begin{ex} 
Standard examples of skew-symmetric bicharacters are
\begin{enumerate}[label=\upshape{\arabic*)}, ref=\upshape{\arabic*}, labelindent=5pt, leftmargin=*]
\item $\Gamma=\mathbb{Z}_2,\quad \varepsilon(i, j)=(-1)^{ij}$,
\item $\Gamma=\mathbb{Z}_2^n=\{(\alpha_1, \dots, \alpha_n) \mid \alpha_i\in\mathbb{Z}_2 \}$,  \\ $\varepsilon((\alpha_1, \dots, \alpha_n), (\beta_1, \dots, \beta_n)) = (-1)^{\alpha_1\beta_1+\dots+\alpha_n\beta_n}$,
\item  $\Gamma=\mathbb{Z}_2\times\mathbb{Z}_2,\quad \varepsilon((i_1, i_2), (j_1, j_2))=(-1)^{i_1j_2-i_2j_1}$,
\item $\Gamma=\mathbb{Z}$,  $\epsilon(m,n)=\cos(mn\pi)=(-1)^{mn}$,
\item  $\Gamma=\mathbb{Z}\times\mathbb{Z} ,\quad \varepsilon((i_1, i_2), (j_1, j_2))=(-1)^{(i_1+i_2)(j_1+j_2)}$,
\item $\Gamma=\{-1, +1\} , \quad\varepsilon(i, j)=(-1)^{(i-1)(j-1)/{4}}$.
\end{enumerate}
\end{ex}
\begin{defi}[\hspace{-0.1pt}\cite{Bakayoko2}]
\label{dfn-BiHomasscoloralg}
(Multiplicative) {\bf BiHom-associative color algebra} $(\mathfrak{g},\mu,\epsilon,\alpha,\beta)$ consist of a $\Gamma$-graded vector space $\mathfrak{g}=\displaystyle\bigoplus\limits_{\gamma\in\Gamma}\mathfrak{g}_\gamma$, two even commuting linear maps $\alpha,\beta:\mathfrak{g}\to\mathfrak{g}$, an even bilinear map $\mu:\mathfrak{g}\times\mathfrak{g}\to\mathfrak{g}$ and a skew-symmetric bicharacter $\epsilon :\Gamma\times \Gamma\rightarrow {\bf \mathbb{K}^*}$ such that for all $x,y,z\in\mathcal{H}(\mathfrak{g})$,
\begin{align}
\mu(\alpha(x),\alpha(y))=\alpha(\mu(x,y)),\quad \mu(\beta(x),\beta(y))=\beta(\mu(x,y)) &\quad \text{multiplicativity},\label{asso-multip-cond}\\
ass_{\alpha,\beta}(x,y,z)=\mu(\alpha(x),\mu(y,z))-\mu(\mu(x,y),\beta(z))=0 &\quad  \epsilon\text{-BiHom-associativity}. \label{BiHom-ass-color-identity}
\end{align}
\end{defi}

\begin{defi} \label{nonmultBiHomasscolalg}
If the multiplicativity conditions \eqref{asso-multip-cond} in 
Definition \ref{dfn-BiHomasscoloralg} are not required, that is, if $\alpha$ and $\beta$ are just commuting linear maps, then we call this algebraic structure \textbf{possibly non-multiplicative BiHom-associative color algebra}, and if the multiplicativity \eqref{asso-multip-cond} is not satisfied (or satisfied), we refer to it as non-multiplicative (or multiplicative, respectively). If linear mappings
$\alpha$ and $\beta$ are not commuting, then we call it {\it a BiHom-associative color algebra with non-commuting twisting maps}.

A morphism $f:(\mathfrak{g}_1,\mu_1,\epsilon,\alpha_1,\beta_1)\to(\mathfrak{g}_2,\mu_2,\epsilon,\alpha_2,\beta_2)$ of the possibly non-multiplicative BiHom-associative color algebras is an even linear map $f:\mathfrak{g}_1\to\mathfrak{g}_2$ such that $\alpha_2\circ f=f\circ\alpha_1,\;\beta_2\circ f=f\circ\beta_1$ and $f\circ\mu_1=\mu_2\circ(f\otimes f)$.

A possibly non-multiplicative BiHom-associative color algebra $(\mathfrak{g},\mu,\epsilon,\alpha,\beta)$ is said to be BiHom-commutative if, for all $x,y\in\mathcal{H}(\mathfrak{g})$,
\begin{equation}\label{bihom-comm-color-ass-alg}
\mu(\beta(x),\alpha(y))=\epsilon(x,y)\mu(\beta(y),\alpha(x)).
\end{equation}\end{defi}

\begin{rmk}
If $\alpha=\beta$, possibly non-multiplicative BiHom-associative color algebras reduce to general Hom-associative color algebras and multiplicative Hom-associative color algebras when $\alpha=\beta$ is moreover multiplicative. For $\alpha=\beta=Id_\mathfrak{g}$, one recovers the associative color algebras.
\end{rmk}
\begin{pro}[\hspace{-0,1pt}\cite{Yaling-Yan}]\label{pro-twist-assoc-color}
Let $(\mathfrak{g},\mu,\epsilon)$ be an associative color algebra and $\alpha,\;\beta:\mathfrak{g}\to\mathfrak{g}$ two even commuting linear maps satisfying condition \eqref{asso-multip-cond}, then the tuple $(\mathfrak{g},\mu_{\alpha,\beta},\epsilon,\alpha,\beta)$ is a BiHom-associative color algebra, where,  for all $x,y\in\mathcal{H}(\mathfrak{g})$,
\begin{equation}\label{twist-bihom-ass}
\mu_{\alpha,\beta}(x,y)=\mu(\alpha(x),\beta(y)).
\end{equation}
\end{pro}
\begin{defi}
Let $(\mathfrak{g},\mu,\epsilon,\alpha,\beta)$ be a BiHom-associative color algebra. A left representation of $\mathfrak{g}$ (also called left $\mathfrak{g}$-module) is a quadruple $(V,\mathfrak l,\alpha_V,\beta_V)$ consisting of a $\Gamma$-graded vector space $V=\displaystyle\bigoplus\limits_{\gamma\in\Gamma}V_\gamma$, an even linear map 
$$\mathfrak l:\mathfrak{g}\otimes V\to V,\;x\otimes u\mapsto x\cdot u=\mathfrak l(x)u$$ and two even linear maps $\alpha_V,\beta_V:V\to V$, such that, for all $x,y\in\mathcal{H}(\mathfrak{g}),\;u\in\mathcal{H}(V)$,
\begin{align}
\alpha_V\circ\beta_V&=\beta_V\circ\alpha_V ,\label{cond-left-mod-asso-color1}\\
\alpha_V(\mathfrak{l}(x\otimes u))&=\alpha(x)\cdot\alpha_V(u),\label{cond-left-mod-asso-color2}\\
\beta_V(\mathfrak{l}(x\otimes u))&=\beta(x)\cdot\beta_V(u).\label{cond-left-mod-asso-color3}
\end{align}
A right representation $(V,\mathfrak r,\alpha_V,\beta_V)$ can be defined in a similar manner.
\end{defi}
\begin{defi}
Let $(\mathfrak{g},\mu,\epsilon,\alpha,\beta)$ be a BiHom-associative color algebra, $V$ be a $\Gamma$-graded vector space, and $\mathfrak l,\;\mathfrak r\;:\;\mathfrak{g}\rightarrow End(V)$,   $\alpha_V,\;\beta_V:V\rightarrow V$ be even linear maps. The tuple $(V,\mathfrak l,\mathfrak r,\alpha_V,\beta_V)$ is called  representation of $\mathfrak{g}$ (also called $\mathfrak{g}$-bimodule) if $(V,\mathfrak l,\alpha_V,\beta_V)$ is a left representation of $\mathfrak{g}$ and $(V,\mathfrak r,\alpha_V,\beta_V)$ is a right representation of $\mathfrak{g}$ such that, for all $x,\;y\in\mathcal{H}(\mathfrak{g})$ and $v\in\mathcal{H}(V)$,
\begin{eqnarray}
\mathfrak l(\alpha(x))\mathfrak l(y)(v)&=&\mathfrak l (\mu(x,y))(\beta_V(v))\label{cond-rep-BiH-ass1}\\
\mathfrak l(\alpha(x))\mathfrak r(y)(v)&=&\epsilon(x,y)\mathfrak r(\beta(y))\mathfrak l(x)(v)\label{cond-rep-BiH-ass2}\\
\mathfrak r(\mu(x,y))\alpha_V(v)&=&\epsilon(x,y)\mathfrak r(\beta(y))\mathfrak r(x)(v)\label{cond-rep-BiH-ass3}
\end{eqnarray}
\end{defi}
\begin{pro}\label{3-BiHom-ass-direct-sum}
Let $(V,\mathfrak l, \mathfrak r, \alpha_V,\beta_V )$ be a bimodule of a BiHom-associative color algebra $(\mathfrak{g},\mu,\epsilon,
\alpha,\beta)$. Then, the direct sum $\mathfrak{g}\oplus V$ of vector spaces $\mathfrak{g}$ and $V$ is turned into a BiHom-associative color algebra by defining the multiplications in $\mathfrak{g}\oplus V$ 
for all $x,y\in\mathcal{H}(\mathfrak{g})$ and $u,v\in\mathcal{H}(V)$ by
\begin{eqnarray}
\mu_{\mathfrak{g}\oplus V}((x+u),(y+v))&=&\mu(x,y)+\mathfrak l(x)v+\epsilon(y,u)\mathfrak r(y)u,\label{mu-direct-sum} \\
(\alpha\oplus\alpha_V)(x+u)&=&\alpha(x)+\alpha_V(u),\label{alpha-direct-sum}\\
(\beta\oplus\beta_V)(x+u)&=&\beta(x)+\beta_V(u).\label{beta-direct-sum}
\end{eqnarray}
This BiHom-associative color algebra, denoted by $\mathfrak{g}\ltimes_{\alpha,\beta}^{\mathfrak l,\mathfrak r} V$, is called semi-direct product of $(\mathfrak{g},\mu,\epsilon,
\alpha,\beta)$ and $(V,\mathfrak l, \mathfrak r, \alpha_V,\beta_V )$.
\end{pro}
\begin{proof}
Let $((x,y,z),(u,v,w))\in\mathcal{H}(\mathfrak g)^3\times\mathcal{H}(V)^3$. Then,
\begin{align*}
&\mu_{\mathfrak{g}\oplus V}\Big( \mu_{\mathfrak{g}\oplus V}(x+u,y+v),(\beta\oplus\beta_V)(z+w)\Big) \\ 
&= \mu_{\mathfrak{g}\oplus V}\Big(\mu(x,y)+\mathfrak l(x)v+\epsilon(y,u)\mathfrak r(y)u,\beta(z)+\beta_V(w)\Big)\\
&=\mu(\mu(x,y),\beta(z))+\mathfrak l(\mu(x,y))\beta_V(w)+\epsilon(x+y,z)\mathfrak r(\beta(z))\Big(\mathfrak l(x)v+\epsilon(y,u)\mathfrak r(y)u\Big), \\ 
&\mu_{\mathfrak{g}\oplus V}\Big((\alpha\oplus\alpha_V)(x+u), \mu_{\mathfrak{g}\oplus V}(y+v,z+w)\Big)\\
&=\mu(\alpha(x),\mu(y,z))+\mathfrak l(\alpha(x))\Big(\mathfrak l(y)w+\epsilon(z,v)\mathfrak r(z)v\Big)+\epsilon(u,y+z)\mathfrak r(\mu(y,z))\alpha_V(u).
\end{align*}
Then, the even product $\mu_{\mathfrak{g}\oplus V}$ defines on $\mathfrak{g}\oplus V$ a BiHom-associative color algebra structure if and only if, 
$$\mu_{\mathfrak{g}\oplus V}\Big((\alpha\oplus\alpha_V)(x+u), \mu_{\mathfrak{g}\oplus V}(y+v,z+w)\Big)=\mu_{\mathfrak{g}\oplus V}\Big( \mu_{\mathfrak{g}\oplus V}(x+u,y+v),(\beta\oplus\beta_V)(z+w)\Big),$$
which by using the fact that $(\mathfrak g,\mu,\alpha,\beta)$ is a BiHom-associative color algebra, and by identification of terms, is equivalent to 
\begin{eqnarray}
\mathfrak l(\mu(x,y))\beta_V(w) &=& \mathfrak l(\alpha(x))\mathfrak l(y)w,\\ \epsilon(x+y,z)\mathfrak r(\beta(z))\mathfrak l(x)v &=& \epsilon(z,v)\mathfrak l(\alpha(x))\mathfrak r(z)v, \\
\epsilon(x+y,z)\epsilon(y,u)\mathfrak r(\beta(z))\mathfrak r(y)u &=& \epsilon(u,y+z)\mathfrak r(\mu(y,z))\alpha_V(u),
\end{eqnarray}
which gives the result.
\end{proof}
\begin{defi}
A \textbf{Kupershmidt operator} on a BiHom-associative color algebra $(\mathfrak{g},\mu,\epsilon,
\alpha,\beta)$  with respect to a bimodule $(V,\mathfrak l, \mathfrak r, \alpha_V,\beta_V )$ is defined as an even linear map $T:V\to\mathfrak{g}$ satisfying, for all $u,v\in\mathcal{H}(V)$,
\begin{eqnarray}
& T\circ\alpha_V=\alpha\circ T,\quad  T\circ\beta_V=\beta\circ T,\label{cond-O-op-BiH-ass-color1} \\
& \mu(T(u),T(v)) = T\Big(\mathfrak l(T(u))v+\epsilon(u,v)\mathfrak r(T(v))u\Big). \label{cond-O-op-BiH-ass-color2}
\end{eqnarray}
\end{defi}
\begin{rmk}
A \textbf{Rota-Baxter operator} $\mathcal{R}$ of weight zero on a BiHom-associative color algebra $(\mathfrak{g},\mu,\epsilon,\alpha,\beta)$ is just a Kupershmidt operator associated with the bimodule $(\mathfrak{g},\mathfrak L,\mathfrak R,\alpha,\beta)$  where $\mathfrak L$ and $\mathfrak R$ are the left and right multiplication operators corresponding to the multiplication $\mu$, that is $\mathcal{R}$ commuting with $\alpha$ and $\beta$ and satisfying, for all $x,y\in\mathcal{H}(\mathfrak g)$,
\begin{equation}
\mu(\mathcal{R}(x),\mathcal{R}(y))=\mathcal{R}\Big(\mu(\mathcal{R}(x),y)+\mu(x,\mathcal{R}(y)\Big). \label{RB-op-BiH-ass-color}
\end{equation}
\end{rmk}
\begin{pro}\label{ind-BiH-ass-color-RB}
Let $\mathcal{R}:\mathfrak{g}\to\mathfrak{g}$ be a Rota-Baxter operator of weight zero on a BiHom associative color algebra $(\mathfrak{g},\mu,\epsilon,\alpha,\beta)$. Then, there exists a BiHom associative color algebra structure $\mu_\mathcal{R}$ on $\mathfrak{g}$ given for any $ x,y\in\mathcal{H}(\mathfrak{g})$ by
\begin{equation}\label{str-op-BiH-ass-color-ind}
\mu_\mathcal{R}(x,y)= \mu(\mathcal{R}(x),y)+\mu(x,\mathcal{R}(y)).
\end{equation}
\end{pro}
\begin{proof}
For $x,y,z\in\mathcal{H}(\mathfrak g)$,
\begin{eqnarray*}
\mu_\mathcal{R}(\alpha(x),\mu_\mathcal{R}(y,z))&=&\mu_\mathcal{R}\big(\alpha(x),\mu(\mathcal{R}(y),z\big)+\mu\big(y,\mathcal{R}(z))\big)\\&=&\mu_\mathcal{R}\big(\alpha(x),\mu(\mathcal{R}(y),z)\big)+\mu_\mathcal{R}\big(\alpha(x),\mu(y,\mathcal{R}(z))\big)\\&=&\mu\big(\mathcal{R}(\alpha(x)),\mu(\mathcal{R}(y),z)\big)+\underbrace{\mu\big(\alpha(x),\mathcal{R}(\mu(\mathcal{R}(y),z))\big)}_{a_1}\\&+&\mu\big(\mathcal{R}(\alpha(x)),\mu(y,\mathcal{R}(z))\big)+\underbrace{\mu\big(\alpha(x),\mathcal{R}(\mu(y,\mathcal{R}(z)))\big)}_{a_2}.
\end{eqnarray*}
Since $\mathcal{R}$ is a Rota-Baxter operator of weight $0$ on $\mathfrak g$, we obtain
\begin{eqnarray*}
a_1+a_2 &=& \mu(\alpha(x),\mu(\mathcal{R}(y),\mathcal{R}(z))), \text{which implies that}\\
\mu_\mathcal{R}\Big(\alpha(x),\mu_\mathcal{R}(y,z)\Big)&=&\mu\Big(\mathcal{R}(\alpha(x)),\mu(\mathcal{R}(y),z)\Big)+\mu\Big(\mathcal{R}(\alpha(x)),\mu(y,\mathcal{R}(z))\Big)\\&+&\mu\Big(\alpha(x),\mu(\mathcal{R}(y),\mathcal{R}(z))\Big)\\&=&\mu\Big(\mu(\mathcal{R}(x),\mathcal{R}(y)),\beta(z)\Big)+\mu\Big(\mu(\mathcal{R}(x),y),\beta(\mathcal{R}(z))\Big)\\&+&\mu\Big(\mu(x,\mathcal{R}(y)),\beta(\mathcal{R}(z))\Big) \\&=&\mu\Big(\mathcal{R}(\mu_{\mathcal{R}}(x,y)),\beta(z)\Big)+\mu\Big(\mu_{\mathcal{R}}(x,y),\mathcal{R}(\beta(z))\Big)\\&=& \mu_{\mathcal{R}}\big(\mu_{\mathcal{R}}(x,y),\beta(z)\big),
\end{eqnarray*}
which proves the result.
\end{proof}
\begin{defi}
{\bf $3$-Lie color algebras} are $\Gamma$-graded vector spaces $\mathfrak{g}=\displaystyle\bigoplus\limits_{\gamma\in\Gamma}\mathfrak{g}_\gamma$ equipped with an even trilinear map $[\cdot,\cdot,\cdot]:\mathfrak{g}\times\mathfrak{g}\times\mathfrak{g}\to\mathfrak{g}$ and a commutation factor (skew-symmetric bicharacter) $\epsilon:\Gamma\times\Gamma\to\mathbb{K}^*=\mathbb{K}\setminus \left\{0\right\}$ such that for any $x,y,z,t,u\in\mathcal{H}(\mathfrak{g})$,
\begin{align}
|[x,y,z]|&=|x|+|y|+|z|,\label{cond-hom}\\
[x,y,z]&=-\epsilon(x,y)[y,x,z]=-\epsilon(y,z)[x,z,y],\quad (\epsilon\text{-skew-symmetry})\label{ColorSkewSym} \\
[x,y,[z,t,u]]&=[[x,y,z],t,u]+\epsilon(x+y,z)[z,[x,y,t],u]\label{color-Jacobi-identity}\\
&+\epsilon(x+y,z+t)[z,t,[x,y,u]], \quad (\epsilon\text{-fundamental identity}). \nonumber
\end{align}
\end{defi}
\begin{defi}\label{Color-3-HomLieAlg}
{\it \bf{A possibly non-multiplicative $3$-Hom-Lie color algebra}} is a quadruple
$(\mathfrak{g},[\cdot,\cdot,\cdot],\epsilon,\alpha)$ consisting of a
$\Gamma$-graded vector space $\mathfrak{g}$, a skew-symmetric bicharacter $\epsilon$, an even trilinear map
$[\cdot,\cdot,\cdot]:\mathfrak{g}\times\mathfrak{g}\times\mathfrak{g}\rightarrow\mathfrak{g}$ {\rm (}that is $([\mathfrak{g}_a,\mathfrak{g}_b,\mathfrak{g}_c]\subseteq\mathfrak{g}_{a+b+c})${\rm )}
and an even linear map $\alpha:\mathfrak{g}\rightarrow\mathfrak{g}$ such that for all $x,y,z,u,v\in\mathcal{H}(\mathfrak{g})$,
\begin{align}
&  [x,y,z]=-\epsilon(x,y)[y,x,z],\;(\epsilon\text{-skew-symmetry})\label{epsilon-antisymm}\\
&  [\alpha(x), \alpha(y), [z,u,v]]=\epsilon(x+y+z,u+v)[ \alpha(u),\alpha(v),[x,y,z]]\label{epsilon-Hom-Nambu}\\
&- \epsilon(x+y,z+v)\epsilon(u,v)
[\alpha(z),\alpha(v),[x,y,u]]\nonumber\\
&+\epsilon(x+y,z+u)[\alpha(z),\alpha(u), [x,y,v]],\; (\epsilon\text{-}3\text{-Hom-Jacobi identity}).\nonumber
\end{align}

{\bf A $3$-Hom-Lie color algebra}
$(\mathfrak{g},[\cdot,\cdot,\cdot],\epsilon,\alpha)$ is said to be multiplicative if it satisfies
$$\alpha([x,y,z])=[\alpha(x),\alpha(y),\alpha(z)],\;\forall x,y,z\in\mathcal{H}(\mathfrak{g}).$$
A multiplicative $3$-Hom-Lie color algebra $(\mathfrak{g},[\cdot,\cdot,\cdot],\epsilon,\alpha)$  is {\bf regular} if $\alpha$  is an even algebra automorphism.
\end{defi}
\begin{rmk}
If $\alpha=id_{\mathfrak{g}}$, we recover $3$-Lie color algebras structures.    
\end{rmk}
\begin{defi}\label{Color-3-BiHomLieAlg}
{\it \bf{A possibly non-multiplicative $3$-BiHom-Lie color algebra}} 
$(\mathfrak{g},[\cdot,\cdot,\cdot],\epsilon,\alpha,\beta)$ consists of a
$\Gamma$-graded vector space $\mathfrak{g}$, a skew-symmetric bicharacter $\epsilon$, an even trilinear map
$[\cdot,\cdot,\cdot]:\mathfrak{g}\times\mathfrak{g}\times\mathfrak{g}\rightarrow\mathfrak{g}$ 
and two even linear maps $\alpha,\beta:\mathfrak{g}\rightarrow\mathfrak{g}$ such that for all $x,y,z,u,v\in\mathcal{H}(\mathfrak{g})$,
\begin{align}
&  [\beta(x),\beta(y),\alpha(z)]=-\epsilon(x,y)[\beta(y),\beta(x),\alpha(z)],\label{Bihom-epsilon-symm1}\\
&  [\beta(x),\beta(y),\alpha(z)]=-\epsilon(y,z)[\beta(x),\beta(z),\alpha(y)],\label{Bihom-epsilon-symm2}\\
&  [\beta^2(x), \beta^2(y), [\beta(z),\beta(u),\alpha(v)]]=\epsilon(x+y+z,u+v)[ \beta^2(u),\beta^2(v),[\beta(x),\beta(y),\alpha(z)]]\label{epsilon-BiHom-Nambu}\\
&- \epsilon(x+y,z+v)\epsilon(u,v)
[\beta^2(z),\beta^2(v),[\beta(x),\beta(y),\alpha(u)]]\nonumber\\
&+\epsilon(x+y,z+u)[\beta^2(z),\beta^2(u), [\beta(x),\beta(y),\alpha(v)]],\; (\epsilon\text{-}3\text{-BiHom-Jacobi identity}).\nonumber
\end{align}
The conditions \eqref{Bihom-epsilon-symm1} and \eqref{Bihom-epsilon-symm2} are called  $\epsilon$-BiHom-skew-symmetry of $[\cdot,\cdot,\cdot]$.
\end{defi}

\noindent \begin{rmk}
\begin{enumerate}[label=\upshape{\arabic*)}, ref=\upshape{\arabic*}, labelindent=*, leftmargin=*]
\item If $\alpha=\beta$ where $\alpha$ is surjective, we recover possibly non-multiplicative 3-Hom-Lie color algebras structures.
\item The $3$-Lie-color algebras can be viewed as $3$-BiHom-Lie color algebras with $\alpha=\beta=id_\mathfrak{g}$   
because the $\epsilon$-$3$-BiHom-Jacobi identity reduces to the $\epsilon$-$3$-Jacobi identity when $\alpha=\beta=id_\mathfrak{g}$.
\item If $\epsilon(x,y)=1$ or $\epsilon(x,y)=(-1)^{|x||y|}$, then the $3$-BiHom-Lie color algebra is a classical $3$-BiHom-Lie algebra or $3$-BiHom-Lie superalgebra.
\end{enumerate}
\end{rmk}
\begin{defi} \label{def-3-BiHom-color-alg}
{\bf $3$-BiHom-Lie color algebra}
$(\mathfrak{g},[\cdot,\cdot,\cdot],\epsilon,\alpha,\beta)$ is said to be multiplicative when
$\alpha$ and $\beta$ satisfy $\alpha\circ\beta=\beta\circ\alpha$, and for all $x,\;y,\;z\in\mathcal{H}(\mathfrak{g})$,
$$\alpha([x,y,z])=[\alpha(x),\alpha(y),\alpha(z)], \quad \beta([x,y,z])=[\beta(x),\beta(y),\beta(z)].$$
A multiplicative $3$-BiHom-Lie color algebra $(\mathfrak{g},[\cdot,\cdot,\cdot],\epsilon,\alpha,\beta)$  is {\bf regular} if $\alpha$ and $\beta$ are even algebra automorphisms.
\end{defi}
\begin{ex}
Let $\Gamma=\mathbb{Z}_2,\;\epsilon(i,j)=(-1)^{ij}$. Let $\mathfrak{g} = \mathfrak{g}_0\oplus\mathfrak{g}_1$ where, $\mathfrak{g}_0=<e_1,e_2>$ and $\mathfrak{g}_1=<e_3>$.
Define the even linear maps $\alpha,\beta:\mathfrak{g}\to\mathfrak{g}$ by the following matrices
$$\alpha=\left(
\begin{array}{ccc}
0 & 1 & 0  \\
0 & 1 & 0  \\
0 & 0 & 1  \\
\end{array}\right),\;\;\;~~~~\beta=\left(
\begin{array}{ccc}
0 & -1 & 0  \\
0 & -1 & 0  \\
0 & 0 & 1  \\
\end{array}\right).$$
Consider the ternary bracket $[\cdot,\cdot,\cdot]:\mathfrak{g}\times\mathfrak{g}\times\mathfrak{g}\to\mathfrak{g}$ defined by  $[e_1,e_2,e_3]=2e_3$, where the other brackets are zero. Then $(\mathfrak{g}, [\cdot, \cdot,\cdot],\epsilon,\alpha,\beta)$ is a non-multiplicative $3$-BiHom-Lie color algebra.
\end{ex}

A morphism $f:(\mathfrak{g},[\cdot,\cdot,\cdot],\epsilon,\alpha,\beta)\to(\mathfrak{g}',[\cdot,\cdot,\cdot]',\epsilon,\alpha',\beta')$ of $3$-BiHom-Lie color algebras is an even linear map $f:\mathfrak{g}\to\mathfrak{g}'$ such that $f\circ\alpha=\alpha'\circ f,\;f\circ\beta=\beta'\circ f$ and $f([x,y,z])=[f(x),f(y),f(z)]'$ for all $x,y,z\in\mathcal{H}(\mathfrak{g})$.
\begin{thm}\label{thm-twist-assoc-color}
If $(\mathfrak{g},[\cdot,\cdot,\cdot],\epsilon)$ be a $3$-Lie color algebra and $\alpha,\beta:\mathfrak{g}\to\mathfrak{g}$ are two even commuting algebra morphisms on $\mathfrak{g}$, then $(\mathfrak{g},[\cdot,\cdot,\cdot]_{\alpha,\beta},\epsilon,\alpha,\beta)$ is a multiplicative $3$-BiHom-Lie color algebra, where for all $x,y,z\in\mathcal{H}(\mathfrak{g})$,
\begin{equation}\label{twist-3-BiHom-Lie}
[x,y,z]_{\alpha,\beta}=[\alpha(x),\alpha(y),\beta(z)].
\end{equation}
\end{thm}
\begin{proof}
It is easy to see that, for all $x,y,z\in\mathcal{H(\mathfrak{g})}$,
\begin{align}
[\beta(x),\beta(y),\alpha(z)]_{\alpha,\beta} &=-\epsilon(x,y)[\beta(y),\beta(x),\alpha(z)]_{\alpha,\beta}\\ {} [\beta(x),\beta(y),\alpha(z)]_{\alpha,\beta} &=-\epsilon(y,z)[\beta(x),\beta(z),\alpha(y)]_{\alpha,\beta}.
\end{align}
For $x,y,z,t,u\in\mathcal{H(\mathfrak{g})}$,
\begin{eqnarray*}
[\beta^2(x),\beta^2(y),[\beta(z),\beta(t),\alpha(u)]_{\alpha,\beta}]_{\alpha,\beta}&=&
[\alpha\beta^2(x),\alpha\beta^2(y),[\beta^2(z),\beta^2(t),\alpha\beta(u)]_{\alpha,\beta}]\\
&=&[\alpha\beta^2(x),\alpha\beta^2(y),[\alpha\beta^2(z),\alpha\beta^2(t),\alpha\beta^2(u)]]\\
&=&\alpha\beta^2\Big([x,y,[z,t,u]]\Big),
\end{eqnarray*}
the $\epsilon$-$3$-BiHom-Jacobi identity is satisfied.
\end{proof}
\begin{defi}
Let $(\mathfrak{g},[\cdot,\cdot,\cdot],\epsilon,\alpha,\beta)$ be a $3$-Bihom-Lie color algebra. A representation of $\mathfrak{g}$ on a $\Gamma$-graded vector space $V$  is a triple $(\rho,\alpha_V,\beta_V)$, where  $\alpha_V,~\beta_V\in End(V)$ are two commuting linear maps, and
$\rho:\mathfrak{g}\times\mathfrak{g}\longrightarrow End(V)$ is an even $\epsilon$-skewsymmetric bilinear map, such that for all $u,v,x,y\in\mathcal{H}(\mathfrak{g})$,
\begin{align}
&\rho(\alpha(u),\alpha(v))\circ \alpha_V=\alpha_V\circ\rho(u,v),\label{cond-rep-3-BiHom-Lie1}\\
&\rho(\beta(u),\beta(v))\circ\beta_V=\beta_V\circ\rho(u,v),\label{cond-rep-3-BiHom-Lie2}\\
&\rho(\alpha\beta(u),\alpha\beta(v))\circ\rho(x,y)=\epsilon(x+y,u+v)\rho(\beta(x),\beta(y))\circ\rho(\alpha(u),\alpha(v))\label{cond-rep-3-BiHom-Lie3}\\
&\quad +\rho([\beta(u),\beta(v),x],\beta(y))\circ\beta_V+\epsilon(x,u+v) \rho(\beta(x),[\beta(u),\beta(v),y])\circ\beta_V,\nonumber\\
&\rho([\beta(u),\beta(v),x],\beta(y))\circ\beta_V
=\epsilon(u,x+v)\rho(\alpha\beta(v),\beta(x))\circ\rho(\alpha(u),y)\label{cond-rep-3-BiHom-Lie4}\\
&\quad +\epsilon(x,u+v)\rho(\beta(x),\alpha\beta(u))\circ\rho(\alpha(v),y)+\rho((\alpha\beta(u),\alpha\beta(v))\circ\rho(x,y).\nonumber
\end{align}
Such a representation of $\mathfrak{g}$ is denoted by $(V,\rho,\alpha_V,\beta_V)$, We also say that, $V$ is a $\mathfrak{g}$-module.
\end{defi}

\begin{pro}\label{struc-3-BiHom-semidirect}
Let $(V,\rho,\alpha_V,\beta_V)$ be a representation of a multiplicative $3$-Bihom-Lie color algebra  $(\mathfrak{g},[\cdot,\cdot,\cdot],\epsilon,\alpha,\beta)$,  where $\alpha$ and $\beta_V$ are bijective. 

Then, $\mathfrak{g}\ltimes_\rho V:=(\mathfrak{g}\oplus V,[\cdot,\cdot,\cdot]_{\mathfrak{g}\oplus V},\epsilon,\alpha\oplus\alpha_V,\beta\oplus\beta_V)$ is a multiplicative $3$-Bihom-Lie color algebra, where $\alpha\oplus\alpha_V,~\beta\oplus\beta_V:\mathfrak{g}\oplus V\longrightarrow\mathfrak{g}\oplus V$ are defined by $(\alpha\oplus\alpha_V)(x+u)=\alpha(x)+\alpha_V(u)$ and $(\beta\oplus\beta_V)(x+u)=\beta(x)+\beta_V(u),$ and the bracket
$[\cdot,\cdot,\cdot]_{\mathfrak{g}\oplus V}$ is defined for all $x,y,z\in \mathcal{H}(\mathfrak{g})$ and $u,v,w\in\mathcal{H}(V)$ by 
\begin{align}
[x+u,y+v,z+w]_{\mathfrak{g}\oplus V}&=[x,y,z]+\rho(x,y)(w)-\epsilon(y,z)\rho(x,\alpha^{-1}\beta(z))(\alpha_V\beta_V^{-1}(v))\label{crochet-direct-sum}\\
&\hspace{3cm} +\epsilon(x,y+z)\rho(y,\alpha^{-1}\beta(z))(\alpha_V\beta_V^{-1}(u)).\nonumber
\end{align} 
We call $\mathfrak{g}\ltimes_\rho V$ the semi-direct product of the
$3$-Bihom-Lie color algebra $(\mathfrak{g},[\cdot,\cdot,\cdot],\alpha,\beta)$ and $V$.
\end{pro}
\begin{proof} Since $\alpha\circ\beta=\beta\circ\alpha$ and $\alpha_V\circ\beta_V=\beta_V\circ\alpha_V$, it follows that $\alpha\oplus\alpha_V$ and $\beta\oplus\beta_V$ commute.

For $u,v,w\in\mathcal{H}(\mathfrak{g})$ and $x,y,z\in \mathcal{H}(V)$, using
\eqref{cond-rep-3-BiHom-Lie1}, we obtain
\begin{align*}
&(\alpha\oplus\alpha_V)[u+x,v+y,w+z]= (\alpha\oplus\alpha_V)\Big([x,y,z]+\rho(x,y)(w)\\
&\hspace{2cm}-\epsilon(y,z)\rho(x,\alpha^{-1}\beta(z))(\alpha_V\beta_V^{-1}(v))+\epsilon(x,y+z)\rho(y,\alpha^{-1}\beta(z))(\alpha_V\beta_V^{-1}(u))\Big)\\
&=\alpha([x,y,z])+\alpha_V\Big(\rho(x,y)(w)-\epsilon(y,z)\rho(x,\alpha^{-1}\beta(z))(\alpha_V\beta_V^{-1}(v))\\
&\hspace{4cm} +\epsilon(x,y+z)\rho(y,\alpha^{-1}\beta(z))(\alpha_V\beta_V^{-1}(u))\Big) \\
&=[\alpha(x),\alpha(y),\alpha(z)]+ \rho(\alpha(x),\alpha(y))\alpha_V(w)-\epsilon(y,z)\rho(\alpha(x),\beta(z))(\alpha_V^2\beta_V^{-1}(v))\\
&\hspace{4cm} +\epsilon(x,y+z)\rho(\alpha(y),\beta(z))(\alpha_V^2\beta_V^{-1}(u))\\
&=[(\alpha\oplus\alpha_V)(u+x),(\alpha\oplus\alpha_V)(v+y),(\alpha\oplus\alpha_V)(w+z)]_{\mathfrak g\oplus V} .
\end{align*}
Similarly, we have $(\beta\oplus\beta_V)[u+x,v+y,w+z]_{\mathfrak g\oplus V}=[(\beta\oplus\beta_V)(u+x),(\beta\oplus\beta_V)(v+y),(\beta\oplus\beta_V)(w+z)]$, which gives the multiplicativity of $[\cdot,\cdot,\cdot]_{\mathfrak g\oplus V}$.

For any $u,v,w\in\mathcal{H}(\mathfrak{g})$ and $x,y,z\in \mathcal{H}(V)$,
\begin{align*}
&[(\beta\oplus\beta_V)(u+x),(\beta\oplus\beta_V)(v+y),(\alpha\oplus\alpha_V)(w+z)]_{\mathfrak{g}\oplus V}\\
&=[\beta(u)+\beta_V(x),\beta(v)+\beta_V(y),\alpha(w)+\alpha_V(z)]_{\mathfrak{g}\oplus V}\\
&=[\beta(u),\beta(v),\alpha(w)]+\rho(\beta(u),\beta(v))(\alpha_V(z))\\
&\quad -\epsilon(v,w)\rho(\beta(u),\alpha^{-1}\beta\alpha(w))(\alpha_V\beta_V^{-1}\beta_V(y))\\
&\quad +\epsilon(u,v+w)\rho(\beta(v),\alpha^{-1}\beta\alpha(w))(\alpha_V\beta_V^{-1}\beta_V(x))\\
&=-\epsilon(u,v)[\beta(v),\beta(u),\alpha(w)]-\epsilon(u,v)\rho(\beta(v),\beta(u))(\alpha_V(z))\\
&\quad +\epsilon(u,v+w)\rho(\beta(v),\alpha^{-1}\beta\alpha(w))(\alpha_V\beta_V^{-1}\beta_V(x))\\
&\quad -\epsilon(v,w)\rho(\beta(u),\alpha^{-1}\beta\alpha(w))(\alpha_V\beta_V^{-1}\beta_V(y))\\
&=-\epsilon(u,v)([\beta(v),\beta(u),\alpha(w)]+\rho(\beta(v),\beta(u))(\alpha_V(z))\\
&\quad -\epsilon(u,w)\rho(\beta(v),\alpha^{-1}\beta\alpha(w))(\alpha_V\beta_V^{-1}\beta_V(x))\\
&\quad +\epsilon(v,u+w)\rho(\beta(u),\alpha^{-1}\beta\alpha(w))(\alpha_V\beta_V^{-1}\beta_V(y)))\\
&=-\epsilon(u,v)[(\beta\oplus\beta_V)(v+y),(\beta\oplus\beta_V)(u+x),(\alpha\oplus\alpha_V)(w+z)]_{\mathfrak{g}\oplus V}
\end{align*}
Similarly, one can show that
\begin{align*}
& [(\beta\oplus\beta_V)(u+x),(\beta\oplus\beta_V)(v+y),(\alpha\oplus\alpha_V)(w+z)]_{\mathfrak{g}\oplus V}\\
& =-\epsilon(v,w)[(\beta\oplus\beta_V)(u+x),(\beta\oplus\beta_V)(w+z),(\alpha\oplus\alpha_V)(v+y)]_{\mathfrak{g}\oplus V}.
\end{align*}
Then, $[\cdot,\cdot,\cdot]_{\mathfrak g\oplus V}$ satisfies conditions \eqref{Bihom-epsilon-symm1} and \eqref{Bihom-epsilon-symm2}.

For
$u_{i}\in\mathfrak{g},~x_{i}\in V,~1\leq i\leq5$,
\begin{align*}
&[(\beta\oplus\beta_V)^{2}(u_{1}+x_{1}),(\beta\oplus\beta_V)^{2}(u_{2}+x_{2}), \\
& \hspace{1,5cm}
[(\beta\oplus\beta_V)(u_{3}+x_{3}),
(\beta\oplus\beta_V)(u_{4}+x_{4}),(\alpha\oplus\alpha_V)(u_{5}+x_{5})]_{\mathfrak{g}\oplus V}]_{\mathfrak{g}\oplus V}\\
&=[\beta^{2}(u_{1})+\beta_V^{2}(x_{1}),\beta^{2}(u_{2})+\beta_V^{2}(x_{2}), \\
& \hspace{1,5cm}[\beta(u_{3})+\beta_V(x_{3}),\beta(u_{4})+\beta_V(x_{4}),\alpha(u_{5})+\alpha_V(x_{5})]_{\mathfrak{g}\oplus V}]_{\mathfrak{g}\oplus V}\\
&=[\beta^{2}(u_{1})+\beta_V^{2}(x_{1}),\beta^{2}(u_{2})+\beta_V^{2}(x_{2}),[\beta(u_{3}),\beta(u_{4}),\alpha(u_{5})]+\rho(\beta(u_{3}),\beta(u_{4}))(\alpha_V(x_{5}))\\
&\quad-\epsilon(u_4,u_5)\rho(\beta(u_{3}),\beta(u_{5}))(\alpha_V(x_{4}))+\epsilon(u_3,u_4+u_5)\rho(\beta(u_{4}),\beta(u_{5}))(\alpha_V(x_{3}))]_{\mathfrak{g}\oplus V}\\
&=[\beta^{2}(u_{1}),\beta^{2}(u_{2}),[\beta(u_{3}),\beta(u_{4}),\alpha(u_{5})]]+\rho(\beta^{2}(u_{1}),\beta^{2}(u_{2}))(\rho(\beta(u_{3}),\beta(u_{4}))(\alpha_V(x_{5}))\\
&\quad-\epsilon(u_4,u_5)\rho(\beta(u_{3}),\beta(u_{5}))(\alpha_V(x_{4}))+\epsilon(u_3,u_4+u_5)\rho(\beta(u_{4}),\beta(u_{5}))(\alpha_V(x_{3})))\\
&\quad-\epsilon(u_2,u_3+u_4+u_5)\rho(\beta^{2}(u_{1}),\alpha^{-1}\beta([\beta(u_{3}),\beta(u_{4}),\alpha(u_{5})]))(\alpha_V\beta_V(x_{2}))\\
&\quad+\epsilon(u_1,u_2+u_3+u_4+u_5)\rho(\beta^{2}(u_{2}),\alpha^{-1}\beta([\beta(u_{3}),\beta(u_{4}),\alpha(u_{5})]))(\alpha_V\beta_V(x_{1}))\\
&=\epsilon(u_1+u_2+u_3,u_4+u_5)[\beta^{2}(u_{4}),\beta^{2}(u_{5}),[\beta(u_{1}),\beta(u_{2}),\alpha(u_{3})]]\\
&\quad-\epsilon(u_1+u_2,u_3+u_5)\epsilon(u_4,u_5)[\beta^{2}(u_{3}),\beta^{2}(u_{5}),[\beta(u_{1}),\beta(u_{2}),\alpha(u_{4})]]\\
&\quad+\epsilon(u_1+u_2,u_3+u_4)[\beta^{2}(u_{3}),\beta^{2}(u_{4}),[\beta(u_{1}),\beta(u_{2}),\alpha(u_{5})]]\\
&\quad+\epsilon(u_1+u_2,u_3+u_4)\rho(\beta^{2}(u_{3}),\beta^{2}(u_{4}))\rho(\beta(u_{1}),\beta(u_{2}))(\alpha_V(x_{5}))\\
&\quad+\rho(\alpha^{-1}\beta([\beta(u_{1}),\beta(u_{2}),\alpha(u_{3})]),\beta^{2}(u_{4}))(\alpha_V\beta_V)(x_{5}))\\
&\quad+\epsilon(u_3,u_1+u_2)\rho(\beta^{2}(u_{3}),\alpha^{-1}\beta([\beta(u_{1}),\beta(u_{2}),\alpha(u_{4})]))(\alpha_V\beta_V(x_{5}))\\
&\quad-\epsilon(u_1+u_2,u_3+u_5)\epsilon(u_4,u_5)\rho(\beta^{2}(u_{3}),\beta^{2}(u_{5}))\rho(\beta(u_{1}),\beta(u_{2}))(\alpha_V(x_{4}))\\
&\quad-\epsilon(u_4,u_5)\rho(\alpha^{-1}\beta([\beta(u_{1}),\beta(u_{2}),\alpha(u_{3})]),\beta^{2}(u_{5}))(\alpha_V)\beta_V)(x_{4}))\\
&\quad-\epsilon(u_1+u_2,u_3)\epsilon(u_4,u_5)\rho(\beta^{2}(u_{3}),\alpha^{-1}\beta([\beta(u_{1}),\beta(u_{2}),\alpha(u_{5})]))(\alpha_V\beta_V(x_{4}))\\
&\quad+\epsilon(u_1+u_2,u_3+u_4)\epsilon(u_3,u_4+u_5)\rho(\beta^{2}(u_{4}),\beta^{2}(u_{5}))\rho(\beta(u_{1}),\beta(u_{2}))(\alpha_V(x_{3}))\\
&\quad+\epsilon(u_3,u_4+u_5)\rho(\alpha^{-1}\beta([\beta(u_{1}),\beta(u_{2}),\alpha(u_{4})]),\beta^{2}(u_{5}))(\alpha_V)\beta_V)(x_{5}))\\
&\quad+\epsilon(u_1+u_2,u_3)\epsilon(u_3,u_4+u_5)\rho(\beta^{2}(u_{4}),\alpha^{-1}\beta([\beta(u_{1}),\beta(u_{2}),\alpha(u_{5})]))(\alpha_V\beta_V(x_{3}))\\
&\quad+\epsilon(u_1+u_2,u_3+u_4+u_5)\rho(\alpha^{-1}\beta([\beta(u_{3}),\beta(u_{4}),\alpha(u_{5})]),\beta^{2}(u_{1}))(\alpha_V\beta_V(x_{2}))\\
&\quad-\epsilon(u_1+u_2,u_3+u_4+u_5)\epsilon(u_1,u_2)\rho(\alpha^{-1}\beta([\beta(u_{3}),\beta(u_{4}),\alpha(u_{5})]),\beta^{2}(u_{2}))(\alpha_V\beta_V(x_{1}))\\
&=\epsilon(u_1+u_2+u_3,u_4+u_5)[\beta^{2}(u_{4}),\beta^{2}(u_{5}),[\beta(u_{1}),\beta(u_{2}),\alpha(u_{3})]]
\\
&\quad-\epsilon(u_1+u_2,u_3+u_5)\epsilon(u_4,u_5)[\beta^{2}(u_{3}),\beta^{2}(u_{5}),[\beta(u_{1}),\beta(u_{2}),\alpha(u_{4})]]\\
&\quad+\epsilon(u_1+u_2,u_3+u_4)[\beta^{2}(u_{3}),\beta^{2}(u_{4}),[\beta(u_{1}),\beta(u_{2}),\alpha(u_{5})]]\\
&\quad+\epsilon(u_1+u_2,u_3+u_4)\rho(\beta^{2}(u_{3}),\beta^{2}(u_{4}))\rho(\beta(u_{1}),\beta(u_{2}))(\alpha_V(x_{5}))\\
&\quad+\rho(\alpha^{-1}\beta([\beta(u_{1}),\beta(u_{2}),\alpha(u_{3})]),\beta^{2}(u_{4}))(\alpha_V)\beta_V)(x_{5}))\\
&\quad+\epsilon(u_3,u_1+u_2)\rho(\beta^{2}(u_{3}),\alpha^{-1}\beta([\beta(u_{1}),\beta(u_{2}),\alpha(u_{4})]))(\alpha_V\beta_V(x_{5}))\\
&\quad-\epsilon(u_1+u_2,u_3+u_5)\epsilon(u_4,u_5)\rho(\beta^{2}(u_{3}),\beta^{2}(u_{5}))\rho(\beta(u_{1}),\beta(u_{2}))(\alpha_V(x_{4}))\\
&\quad-\epsilon(u_4,u_5)\rho(\alpha^{-1}\beta([\beta(u_{1}),\beta(u_{2}),\alpha(u_{3})]),\beta^{2}(u_{5}))(\alpha_V)\beta_V)(x_{4}))\\
&\quad-\epsilon(u_1+u_2,u_3)\epsilon(u_4,u_5)\rho(\beta^{2}(u_{3}),\alpha^{-1}\beta([\beta(u_{1}),\beta(u_{2}),\alpha(u_{5})]))(\alpha_V\beta_V(x_{4}))\\
&\quad+\epsilon(u_1+u_2,u_3+u_4)\epsilon(u_3,u_4+u_5)\rho(\beta^{2}(u_{4}),\beta^{2}(u_{5}))\rho(\beta(u_{1}),\beta(u_{2}))(\alpha_V(x_{3}))\\
&\quad+\epsilon(u_3,u_4+u_5)\rho(\alpha^{-1}\beta([\beta(u_{1}),\beta(u_{2}),\alpha(u_{4})]),\beta^{2}(u_{5}))(\alpha_V)\beta_V)(x_{3}))\\
&\quad+\epsilon(u_1+u_2,u_3)\epsilon(u_3,u_4+u_5)\rho(\beta^{2}(u_{4}),\alpha^{-1}\beta([\beta(u_{1}),\beta(u_{2}),\alpha(u_{5})]))(\alpha_V\beta_V(x_{3}))\\
&\quad+\epsilon(u_1+u_2+u_3,u_4+u_5)\epsilon(u_3,u_1+u_2)\rho(\beta^{2}(u_{4}),\beta^{2}(u_{5}))\rho(\beta(u_{3}),\beta(u_{1}))(\alpha_V(x_{2}))\\
&\quad+\epsilon(u_1+u_2,u_3+u_4+u_5)\epsilon(u_5,u_3+u_4)\rho(\beta^{2}(u_{5}),\beta^{2}(u_{3}))\rho(\beta(u_{4}),\beta(u_{1}))(\alpha_V(x_{2}))\\
&\quad+\epsilon(u_1+u_2,u_3+u_4+u_5)\rho(\beta^{2}(u_{3}),\beta^{2}(u_{4}))\rho(\beta(u_{5}),\beta(u_{1}))(\alpha_V(x_{2}))\\
&\quad-\epsilon(u_1+u_2+u_3,u_4+u_5)\epsilon(u_1,u_2+u_3)\rho(\beta^{2}(u_{4}),\beta^{2}(u_{5}))\rho(\beta(u_{3}),\beta(u_{2}))(\alpha_V(x_{1}))\\
&\quad-\epsilon(u_1+u_2,u_3+u_4+u_5)\epsilon(u_5,u_3+u_4)\epsilon(u_1,u_2)\rho(\beta^{2}(u_{5}),\beta^{2}(u_{3}))\rho(\beta(u_{4}),\beta(u_{2}))(\alpha_V(x_{1}))\\
&\quad-\epsilon(u_1+u_2,u_3+u_4+u_5)\rho(\beta^{2}(u_{3}),\beta^{2}(u_{4}))\rho(\beta(u_{5}),\beta(u_{2}))(\alpha_V(x_{1}))\\
&=\epsilon(u_1+u_2+u_3,u_4+u_5)[\beta^{2}(u_{4}),\beta^{2}(u_{5}),[\beta(u_{1}),\beta(u_{2}),\alpha(u_{3})]]\\
&\quad+\rho(\beta^{2}(u_{4}),\beta^{2}(u_{5}))(\epsilon(u_1+u_2+u_3,u_4+u_5)\rho(\beta(u_{1}),\beta(u_{2}))(\alpha_V(x_{3}))\\
&\quad-\epsilon(u_1+u_2+u_3,u_4+u_5)\epsilon(u_2,u_3)\epsilon(u_1,u_2)\rho(\beta(u_{1}),\beta(u_{3}))(\alpha_V(x_{2}))\\
&\quad+\epsilon(u_1+u_2+u_3,u_4+u_5)\epsilon(u_1,u_2+u_3)\rho(\beta(u_{2}),\beta(u_{3}))(\alpha_V(x_{1})))\\
&\quad-\epsilon(u_1+u_2+u_3,u_4)\rho(\beta^{2}(u_{4}),\alpha^{-1}\beta([\beta(u_{1}),\beta(u_{2}),\alpha(u_{3})]))(\alpha_V\beta_V(x_{5}))\\
&\quad-\epsilon(u_1+u_2+u_3+u_4,u_5)\rho(\beta^{2}(u_{5}),\alpha^{-1}\beta([\beta(u_{1}),\beta(u_{2}),\alpha(u_{3})]))(\alpha_V\beta_V(x_{4}))\\
&\quad-\epsilon(u_1+u_2,u_3+u_5)\epsilon(u_4,u_5)[\beta^{2}(u_{3}),\beta^{2}(u_{5}),[\beta(u_{1}),\beta(u_{2}),\alpha(u_{4})]]\\
&\quad-\rho(\beta^{2}(u_{3}),\beta^{2}(u_{5}))(\epsilon(u_1+u_2,u_3+u_5)\epsilon(u_4,u_5)(\rho(\beta(u_{1}),\beta(u_{2}))(\alpha_V(x_{4}))\\
&\quad-\epsilon(u_1+u_2,u_3+u_5)\epsilon(u_4,u_2+u_5)\rho(\beta(u_{1}),\beta(u_{4}))(\alpha_V(x_{2}))\\
&\quad+\epsilon(u_1+u_2,u_3+u_5)\epsilon(u_4,u_2+u_5)\epsilon(u_1,u_2)\rho(\beta(u_{2}),\beta(u_{4}))(\alpha_V(x_{1})))\\
&\quad+\epsilon(u_1+u_2,u_3)\rho(\beta^{2}(u_{3}),\alpha^{-1}\beta([\beta(u_{1}),\beta(u_{2}),\alpha(u_{4})]))(\alpha_V\beta_V(x_{5}))\\
&\quad-\epsilon(u_1+u_2+u_3+u_4,u_5)\epsilon(u_3,u_4)\rho(\beta^{2}(u_{5}),\alpha^{-1}\beta([\beta(u_{1}),\beta(u_{2}),\alpha(u_{4})]))(\alpha_V\beta_V(x_{3}))\\
&\quad+\epsilon(u_1+u_2,u_3+u_4)[\beta^{2}(u_{3}),\beta^{2}(u_{4}),[\beta(u_{1}),\beta(u_{2}),\alpha(u_{5})]]\\
&\quad+\rho(\beta^{2}(u_{3}),\beta^{2}(u_{4}))(\epsilon(u_1+u_2,u_3+u_4)\rho(\beta(u_{1}),\beta(u_{2}))(\alpha_V(x_{5}))\\
&\quad-\epsilon(u_1+u_2,u_3+u_4)\epsilon(u_2,u_5)\rho(\beta(u_{1}),\beta(u_{5}))(\alpha_V(x_{2}))\\
&\quad+\epsilon(u_1+u_2,u_3+u_4)\epsilon(u_1,u_5)\rho(\beta(u_{2}),\beta(u_{5}))(\alpha_V(x_{1}))\\
&\quad-\epsilon(u_1+u_2,u_3)\epsilon(u_4,u_5)\rho(\beta^{2}(u_{3}),\alpha^{-1}\beta([\beta(u_{1}),\beta(u_{2}),\alpha(u_{5})]))(\alpha_V\beta_V(x_{4}))\\
&\quad+\epsilon(u_1+u_2,u_4)\epsilon(u_3,u_4+u_5)\rho(\beta^{2}(u_{4}),\alpha^{-1}\beta([\beta(u_{1}),\beta(u_{2}),\alpha(u_{4})]))(\alpha_V\beta_V(x_{3}))
\\
&=\epsilon(u_1+u_2+u_3,u_4+u_5)[(\beta\oplus\beta_V)^{2}(u_{4}+x_{4}),(\beta\oplus\beta_V)^{2}(u_{5}+x_{5}),\\
&\hspace{2cm} [(\beta\oplus\beta_V)(u_{1}+x_{1}),
(\beta\oplus\beta_V)(u_{2}+x_{2}),(\alpha\oplus\alpha_V)(u_{3}+x_{3})]_{\mathfrak{g}\oplus V}]_{\mathfrak{g}\oplus V}\\
&\quad-\epsilon(u_1+u_,u_3+u_5)\epsilon(u_4,u_5)[(\beta\oplus\beta_V)^{2}(u_{3}+x_{3}),(\beta\oplus\beta_V)^{2}(u_{5}+x_{5}),\\
&\hspace{2cm} [(\beta\oplus\beta_V)(u_{1}+x_{1}),
(\beta\oplus\beta_V)(u_{2}+x_{2}),(\alpha\oplus\alpha_V)(u_{4}+x_{4})]_{\mathfrak{g}\oplus V}]_{\mathfrak{g}\oplus V}\\
&\quad+\epsilon(u_1+u_2,u_3+u_4)[(\beta\oplus\beta_V)^{2}(u_{3}+x_{3}),(\beta\oplus\beta_V)^{2}(u_{4}+x_{4}),[(\beta\oplus\beta_V)(u_{1}+x_{1}),
\\
&(\beta\oplus\beta_V)(u_{2}+x_{2}),(\alpha\oplus\alpha_V)(u_{5}+x_{5})]_{\mathfrak{g}\oplus V}]_{\mathfrak{g}\oplus V}.
\end{align*}
Then, the condition \eqref{epsilon-BiHom-Nambu} is satisfied by $[\cdot,\cdot,\cdot]_{\mathfrak g\oplus V},\;\alpha\oplus\alpha_V$ and $\beta\oplus\beta_V$, which implies that   $\mathfrak{g}\ltimes_\rho V:=(\mathfrak{g}\oplus V,[\cdot,\cdot,\cdot]_{\mathfrak g\oplus V},\alpha\oplus\alpha_V,\beta\oplus\beta_V)$ is a $3$-Bihom-Lie color algebra.
\end{proof}
In the following, we introduce the notion of $\alpha^r\beta^s$-adjoint representation of a regular $3$-BiHom-Lie color algebra.
\begin{defi}
For any integer $r,s$, the {\it \bf{$\alpha^r\beta^s$-adjoint representation}} of a $3$-BiHom-Lie color algebra $(\mathfrak{g},[\cdot,\cdot,\cdot],\epsilon,\alpha,\beta)$, denoted by $ad_{r,s}$, is defined for all $x,y,z\in\mathcal{H}(\mathfrak{g})$ by
$$ad_{r,s}(x,y)(z)=[\alpha^r\beta^s(x),\alpha^r\beta^s(y),z].$$
\end{defi}
\begin{rmk}
For $(r,s)=(0,0)$ we have the adjoint representation of a $3$-BiHom-Lie color algebra denoted simply by $ad$.
\end{rmk}

\begin{pro}
The $\alpha^r\beta^s$-adjoint representation of a regular multiplicative $3$-BiHom-Lie color algebra $(\mathfrak{g},[\cdot,\cdot,\cdot],\epsilon,\alpha,\beta)$ is satisfying conditions \eqref{cond-rep-3-BiHom-Lie1}-\eqref{cond-rep-3-BiHom-Lie4}.
\end{pro}
\begin{proof}
First, $ad_{r,s}$ satisfies condition \eqref{cond-rep-3-BiHom-Lie1} since for $x,y,z\in\mathcal{H}(\mathfrak{g})$,
\begin{eqnarray*}
ad_{r,s}(\alpha(x),\alpha(y))\circ\alpha(z)&=&[\alpha^{r+1}\beta^s(x),\alpha^{r+1}\beta^s(y),\alpha(z)]\\&=&\alpha([\alpha^r\beta^s(x),\alpha^r\beta^s(y),z])\\&=& \alpha\circ ad_{r,s}(x,y)(z),
\end{eqnarray*}
Similarly, one can show that $ad_{r,s}$ satisfies condition \eqref{cond-rep-3-BiHom-Lie2}.

For $x,y,z,u,v\in\mathcal{H}(\mathfrak{g})$, by applying   \eqref{Bihom-epsilon-symm1}-\eqref{epsilon-BiHom-Nambu}, we obtain
\begin{align*}
&ad_{r,s}(\alpha\beta(u),\alpha\beta(v))\circ ad_{r,s}(x,y)(t)=[\alpha^{r+1}\beta^{s+1}(u),\alpha^{r+1}\beta^{s+1}(v),[\alpha^r\beta^s(x),\alpha^r\beta^s(y),t]]\\
&=-\epsilon(y,t)[\alpha^{r+1}\beta^{s+1}(u),\alpha^{r+1}\beta^{s+1}(v),[\alpha^r\beta^s(x),\alpha^{-1}\beta(t),\alpha^{r+1}\beta^{s-1}(y)]]\\
&=-\epsilon(y+t,x+u+v)[\alpha^{-1}\beta^2(t),\alpha^r\beta^{s+1}(y),[\alpha^{r+1}\beta^s(u),\alpha^{r+1}\beta^s(v),\alpha^{r+1}\beta^{s-1}(x)]]\\
&\quad +\epsilon(x+y,u+v+)[\alpha^r\beta^{s+1}(x),\alpha^r\beta^{s+1}(y),[\alpha^{r+1}\beta^s(u),\alpha^{r+1}\beta^s(v),t]]\\
&\quad-\epsilon(u+v,t+x)\epsilon(y,t)[\alpha^r\beta^{s+1}(x),\alpha^{-1}\beta(t),[\alpha^{r+1}\beta^s(u),\alpha^{r+1}\beta^s(v),\alpha^{r+1}\beta^{s-1}(y)]]\\
&=[[\alpha^r\beta^{s+1}(u),\alpha^r\beta^{s+1}(v),\alpha^r\beta^s(x)],\alpha^r\beta^{s+1}(y),\beta(t)]\\
&\quad+\epsilon(x+y,u+v)[\alpha^r\beta^{s+1}(x),\alpha^r\beta^{s+1}(y),[\alpha^{r+1}\beta^s(u),\alpha^{r+1}\beta^s(v),t]])\\
&\quad+\epsilon(x,u+v)[\alpha^r\beta^{s+1}(x),[\alpha^r\beta^{s+1}(u),\alpha^r\beta^{s+1}(v),\alpha^r\beta^s(y)],t]\\
&=
ad_{r,s}([\beta(u),\beta(v),x],\beta(y))\beta(t)+\epsilon(x+y,u+v)ad_{r,s}(\beta(x),\beta(y))ad_{r,s}(\alpha(u),\alpha(v))(t)\\
&\quad+\epsilon(x,u+v)ad_{r,s}(\beta(x),[\beta(u),\beta(v),y])\beta(t),
\end{align*}
and hence $ad_{r,s}$ satisfies condition \eqref{cond-rep-3-BiHom-Lie3}. In the same way,
\begin{align*}
& ad_{r,s}([\beta(u),\beta(v),x],\beta(y))\circ\beta(t)=[[\alpha^r\beta^{s+1}(u),\alpha^r\beta^{s+1}(v),\alpha^r\beta^s(x)],\alpha^r\beta^{s+1}(y),\beta(t)]\\
&=\varepsilon(y+t,x+u+v) [\alpha^r\beta^{s+1}(y),\alpha^{-1}\beta^2(t),[\alpha^{r+1}\beta^s(u),\alpha^{r+1}\beta^s(v),\alpha^{r+1}\beta^{s-1}(x)]]\\
&=\varepsilon(u,x+y+t+v)[\alpha^{r+1}\beta^{s+1}(v),\alpha^r\beta^{s+1}(x),[\alpha^r\beta^s(y),\alpha^{-1}\beta(t),\alpha^{r+2}\beta^{s-1}(u)]]\\
&\quad- \varepsilon(v,x+y+t)[\alpha^{r+1}\beta^{s+1}(u),\alpha^r\beta^{s+1}(x),[\alpha^r\beta^s(y),\alpha^{-1}\beta(t),\alpha^{r+2}\beta^{s-1}(v)]] \\
&\quad+\varepsilon(x,y+t)[\alpha^{r+1}\beta^{s+1}(u),\alpha^{r+1}\beta^{s+1}(v),[\alpha^r\beta^s(y),\alpha^{-1}\beta(t),\alpha^{r+1}\beta^{s-1}(x)]]\\
&=\varepsilon(u,x+v)[\alpha^{r+1}\beta^{s+1}(v),\alpha^r\beta^{s+1}(x),[\alpha^{r+1}\beta^s(u),\alpha^r\beta^s(y),t]]\\
&\quad- \varepsilon(x,u+v)[\alpha^{r+1}\beta^{s+1}(x),\alpha^r\beta^{s+1}(u),[\alpha^{r+1}\beta^s(v),\alpha^r\beta^s(y),t]] \\
&\quad+[\alpha^{r+1}\beta^{s+1}(u),\alpha^{r+1}\beta^{s+1}(v),[\alpha^r\beta^s(x),\alpha^r\beta^s(y),t]]\\
&=\varepsilon(u,x+v)ad_{r,s}(\alpha\beta(v),\beta(x))\circ ad_{r,s}(\alpha(u),y)(t)\\
&\quad-\varepsilon(x,u+v)ad_{r,s}(\alpha\beta(x),\beta(u))\circ ad_{r,s}(\alpha(v),y)(t)\\
&\quad+ad_{r,s}(\alpha\beta(u),\alpha\beta(v))ad_{r,s}(x,y)(t),
\end{align*}
yields that $ad_{r,s}$ satisfies condition \eqref{cond-rep-3-BiHom-Lie4}. The proposition is proved.
\end{proof}
In the following, we introduce the notion of Kupershmidt operators on $3$-BiHom-Lie color algebras.
\begin{defi}
Let $(\mathfrak{g},[\cdot,\cdot,\cdot],\epsilon,\alpha,\beta)$ be a $3$-BiHom-Lie color algebra and $(V,\rho,\alpha_V,\beta_V)$ be a representation of $\mathfrak{g}$. Assume that $\alpha_V$ and $\beta_V$ are bijective. An even linear map $T:V\to\mathfrak{g}$ is called {\it \bf{Kupershmidt operator}} on $\mathfrak{g}$ with respect to the representation $(V,\rho,\alpha_V,\beta_V)$ if it satisfies, for all 
$u,v,w\in\mathcal{H}(V)$,
\begin{align}
\alpha\circ T=T\circ\alpha_V\;\;&\text{and}\;\;\beta\circ T=T\circ\beta_V,\label{cond-O-op-3-BiH-Lie-color1}\\
[T(u),T(v),T(w)]&=T\Big(\rho(T(u),T(v))w-\epsilon(v,w)\rho(T(u),T(\alpha_V^{-1}\beta_V(w)))\alpha_V\beta_V^{-1}(v)\label{cond-O-op-3-BiH-Lie-color2}\\&+\epsilon(u,v+w)\rho(T(v),T(\alpha_V^{-1}\beta_V(w)))\alpha_V\beta_V^{-1}(u)\Big).\nonumber
\end{align}
\end{defi}
\begin{ex}
Let $(\mathfrak{g},[\cdot,\cdot,\cdot],\epsilon,\alpha,\beta)$ be a $3$-BiHom-Lie color algebra.
A Rota-Baxter operator of weight zero on $\mathfrak{g}$ is just a Kupershmidt operator on $\mathfrak{g}$ with respect to the representation $(\mathfrak{g},ad,\alpha,\beta)$, that is, $\mathcal{R}$ commuting with $\alpha$ and $\beta$ satisfying, for all $x,y,z\in\mathcal{H}(\mathfrak{g})$, 
\begin{equation}
[\mathcal{R}(x),\mathcal{R}(y),\mathcal{R}(z)]=\mathcal{R}\Big([\mathcal{R}(x),\mathcal{R}(y),z]+[\mathcal{R}(x),y,\mathcal{R}(z)]+[x,\mathcal{R}(y),\mathcal{R}(z)]\Big).\label{RB-op-3-BiH-Lie-color}
\end{equation}
\end{ex}
\begin{thm}
An even linear map $T:V\to\mathfrak{g}$ is a Kupershmidt operator on a $3$-BiHom-Lie color algebra $(\mathfrak{g},[\cdot,\cdot,\cdot],\epsilon,\alpha,\beta)$ with respect to a representation $(V,\rho,\alpha_V,\beta_V)$ if and only if, the graph $Gr(T)=\{(T(u),u),\;u\in\mathcal{H}(V)\}$ is a subalgebra of the smi-direct product $3$-BiHom-Lie color algebra $(\mathfrak{g}\oplus V,[\cdot,\cdot,\cdot]_{\mathfrak{g}\oplus V},\epsilon,\alpha\oplus\alpha_V,\beta\oplus\beta_V)$ defined in Proposition \ref{struc-3-BiHom-semidirect}.
\end{thm}
\begin{proof}
Let $T:V\to\mathfrak{g}$ be an even linear map. For all $u,v,w\in\mathcal{H}(V)$, 
\begin{multline*}
[(T(u),u),(T(v),v),(T(w),w)]= \\
\Big([T(u),T(v),T(w)]+\rho(T(u),T(v))w-\epsilon(v,w)\rho(T(u),T(\alpha_V^{-1}\beta_V(w)))\alpha_V\beta_V^{-1}(v)\\
+\epsilon(u,v+w)\rho(T(v),T(\alpha_V^{-1}\beta_V(w)))\alpha_V\beta_V^{-1}(u)\Big),
\end{multline*}
which implies that the graph $Gr(T)$ is a subalgebra of the  semi-direct product $3$-BiHom-Lie color algebra $(\mathfrak{g}\oplus V,[\cdot,\cdot,\cdot]_{\mathfrak{g}\oplus V},\epsilon,\alpha\oplus\alpha_V,\beta\oplus\beta_V)$, if and only if $T$ satisfies
\begin{align*}
[T(u),T(v),T(w)]&=T\Big(\rho(T(u),T(v))w-\epsilon(v,w)\rho(T(u),T(\alpha_V^{-1}\beta_V(w)))\alpha_V\beta_V^{-1}(v)\\&+\epsilon(u,v+w)\rho(T(v),T(\alpha_V^{-1}\beta_V(w)))\alpha_V\beta_V^{-1}(u)\Big),
\end{align*}
for all $u,v,w\in\mathcal{H}(V)$, which yields that $T$ is a Kupershmidh operator on $\mathfrak{g}$ with respect to the representation $(V,\rho,\alpha_V,\beta_V)$.
\end{proof}

\begin{cor}
Let $T:V\to\mathfrak{g}$ be a Kupershmidt operator on a $3$-BiHom-Lie color algebra $(\mathfrak{g},[\cdot,\cdot,\cdot],\epsilon,\alpha,\beta)$ with respect to a representation $(V,\rho,\alpha_V,\beta_V)$. Then, there exists a $3$-BiHom-Lie color algebra structure $[\cdot,\cdot,\cdot]_T$ on $V$ given for all $u,v,w\in\mathcal{H}(V)$ by
\begin{align}
[u,v,w]_T=&\rho(T(u),T(v))w-\epsilon(v,w)\rho(T(u),T(\alpha_V^{-1}\beta_V(w)))\alpha_V\beta_V^{-1}(v)\label{3-BiH-Lie-color-kup}\\&+\epsilon(u,v+w)\rho(T(v),T(\alpha_V^{-1}\beta_V(w)))\alpha_V\beta_V^{-1}(u).\nonumber
\end{align}
\end{cor}
\begin{cor}\label{ind-3-BiH-Lie-color-RB}
Let $\mathcal{R}:\mathfrak{g}\to\mathfrak{g}$ be Rota-Baxter operator of weight $0$ on a $3$-BiHom-Lie color algebra $(\mathfrak{g},[\cdot,\cdot,\cdot],\epsilon,\alpha,\beta)$. Then, there exists a $3$-BiHom-Lie color algebra structure $[\cdot,\cdot,\cdot]_\mathcal{R}$ on $\mathfrak{g}$ given for all $x,y,z\in\mathcal{H}(\mathfrak{g})$ by
\begin{equation}\label{3-BiH-Lie-color-RB}
[x,y,z]_\mathcal{R}= [\mathcal{R}(x),\mathcal{R}(y),z]+[\mathcal{R}(x),y,\mathcal{R}(z)]+[x,\mathcal{R}(y),\mathcal{R}(z)].
\end{equation}
\end{cor}
\section{Noncommutative $3$-BiHom-Poisson color algebras} 
\label{sec-noncom3bihompoissoncolalgs}

In this section, we introduce the notion of noncommutative $3$-BiHom-Poisson color algebras and present some properties and related results. The representation theory and Kupershmidt operators of noncommutative $3$-BiHom-Poisson color algebras are also introduced.
\subsection{Noncommutative $3$-BiHom-Poisson color algebras and their representations}
\label{subsec-noncom3bihompoissoncolalgsreps}
\begin{defi}
{\it \bf{A noncommutative $3$-BiHom-Poisson color algebra}} is defined as a $6$-tuple $(\mathfrak{g},\{\cdot,\cdot,\cdot\},\mu,\epsilon,\alpha,\beta)$ such that
\begin{enumerate}[label=\upshape{\arabic*.}, ref=\upshape{\arabic*}, labelindent=5pt, leftmargin=*]
\item $(\mathfrak{g},\mu,\epsilon,\alpha,\beta)$ is a BiHom-associative color algebra.
\item $(\mathfrak{g},\{\cdot,\cdot,\cdot\},\epsilon,\alpha,\beta)$ is a $3$-BiHom-Lie color algebra.
\item  The BiHom-Leibniz rule holds, i.e.,
\begin{equation}\label{BiHom-Leibniz-role}
\{\alpha\beta(x),\alpha\beta(y),\mu(z,t)\}=\mu(\{\beta(x),\beta(y),z\},\beta(t))+\epsilon(x+y,z)\mu(\beta(z),\{\alpha(x),\alpha(y),t\}).
\end{equation}
\end{enumerate}

A $3$-BiHom-Poisson color algebra is a noncommutative $3$-BiHom-Poisson color algebra $(\mathfrak{g},\{\cdot,\cdot,\cdot\},\mu,\epsilon,\alpha,\beta)$ in which $(\mathfrak{g},\mu,\epsilon,\alpha,\beta)$ is a BiHom-commutative BiHom-associative color algebra.\\
\begin{rmk}\
\begin{enumerate}[label=\upshape{\arabic*.}, ref=\upshape{\arabic*}, labelindent=5pt, leftmargin=*]
\item  If $\alpha=\beta=Id_\mathfrak{g}$, we recover the $3$-Poisson color algebra structure.
\item If $\alpha=\beta$ where $\alpha$ is surjective, we recover the noncommutative $3$-Hom-Poisson color algebras structures.
\item If the twisting maps $\alpha$ and $\beta$ are bijective, we call $(\mathfrak{g},\{\cdot,\cdot,\cdot\},\mu,\epsilon,\alpha,\beta)$ regular.
\end{enumerate}
\end{rmk}
\begin{ex}
Let $\Gamma=\mathbb{Z}_2,\;\epsilon(i,j)=(-1)^{ij}$. Consider the two-dimensional graded vector space $\mathfrak{g} = \mathfrak{g}_0\oplus\mathfrak{g}_1=<e_1>\oplus<e_2>$. Define the even linear maps $\alpha,\beta:\mathfrak{g}\to\mathfrak{g}$ by $$\alpha(e_1)=e_1,\;\alpha(e_2)=0\;\;\text{and}\;\;\beta(e_1)=e_1,\;\beta(e_2)=-e_2,$$ and the bracket $\{\cdot,\cdot,\cdot\}:\mathfrak{g}\times\mathfrak{g}\times\mathfrak{g}\to\mathfrak{g}$ by 
$\{e_1,e_1,e_2\}=e_2$ with the other brackets being zeros. Then $(\mathfrak{g}, [\cdot,\cdot, \cdot],\alpha,\beta)$ is a multiplicative $3$-BiHom-Lie superalgebra.
Let $\mu:\mathfrak{g}\times\mathfrak{g}\to\mathfrak{g}$ be an even bilinear map defined by
$$\mu(e_1,e_1)=e_1.$$
Then $(\mathfrak{g},\{\cdot,\cdot,\cdot\},\mu,\epsilon,\alpha,\beta)$ is a $3$-BiHom-Poisson color algebra.
\end{ex}
\end{defi}
\begin{thm}
If $(\mathfrak{g},\{\cdot,\cdot,\cdot\},\mu,\epsilon)$ is a noncommutative $3$-Poisson color algebra and $\alpha,\beta:\mathfrak{g}\to\mathfrak{g}$ are two even commuting algebra morphisms on $\mathfrak{g}$, 
then $(\mathfrak{g},\{\cdot,\cdot,\cdot\}_{\alpha,\beta},\mu_{\alpha,\beta},\epsilon,\alpha,\beta)$ is a noncommutative $3$-BiHom-Poisson color algebra where $\mu_{\alpha,\beta}$ and $\{\cdot,\cdot,\cdot\}_{\alpha,\beta}$ are defined respectively by   \eqref{bihom-comm-color-ass-alg} and \eqref{twist-3-BiHom-Lie}.
\end{thm}
\begin{proof}
According to Proposition \ref{pro-twist-assoc-color} and Theorem \ref{thm-twist-assoc-color}, we deduce that $\mu_{\alpha,\beta}$ defines a BiHom-associative color algebra structure on $\mathfrak{g}$ and $\{\cdot,\cdot,\cdot\}_{\alpha,\beta}$ defines a $3$-BiHom-Lie color algebra structure on $\mathfrak{g}$. It remains to be proved that the BiHom-Leibniz identity is satisfied. For $x,y,z,t\in\mathcal{H}(\mathfrak{g})$, it holds that 
\begin{align*}
& \{\alpha\beta(x),\alpha\beta(y),\mu_{\alpha,\beta}(z,t)\}_{\alpha,\beta}=
\{\alpha^2\beta(x),\alpha^2\beta(y),\beta\mu(\alpha(z),\beta(t))\} \\
&= \{\alpha^2\beta(x),\alpha^2\beta(y),\mu(\alpha\beta(z),\beta^2(t))\}\\
&=\mu(\{\alpha^2\beta(x),\alpha^2\beta(y),\alpha\beta(z)\},\beta^2(t))+\epsilon(x+y,z) \mu(\alpha\beta(z),
\{\alpha^2\beta(x),\alpha^2\beta(y),\beta^2(t)\}) \\
&= \mu_{\alpha,\beta}(\{\beta(x),\beta(y),z\}_{\alpha,\beta},\beta(t))+\epsilon(x+y,z)\mu_{\alpha,\beta}(\beta(z),
\{\alpha(x),\alpha(y),t\}_{\alpha,\beta}),
\end{align*}
and thus the identity \eqref{BiHom-Leibniz-role} is satisfied. The theorem is proved.
\end{proof}
\begin{defi}
A \textbf{representation} $(V,\rho,\mathfrak l,\mathfrak r,\alpha_V,\beta_V)$ 
of the noncommutative $3$-BiHom-Poisson color algebra $(\mathfrak{g}, \{\cdot,\cdot,\cdot \},\mu,\epsilon,\alpha,\beta)$ 
 consist of a $\Gamma$-graded vector space $V$, two even endomorphisms $\alpha_V,\beta_V\in End(V)$,
two even linear maps $\mathfrak l,\mathfrak r:\mathfrak{g}\to End(V)$ and an even $\epsilon$-skew-symmetric bilinear map $\rho:\mathfrak{g}\times\mathfrak{g}\to End(V)$ such that
\begin{enumerate}[label=\upshape{\arabic*.}, ref=\upshape{\arabic*}, labelindent=5pt, leftmargin=*]
\item $(V,\mathfrak l,\mathfrak r,\alpha_V)$ is a bimodule of the BiHom-associative color algebra $(\mathcal{A},\mu,\epsilon,\alpha,\beta)$;
\item $(V,\rho,\alpha_V,\beta_V)$ is a representation of the $3$-BiHom-Lie color algebra
$(\mathfrak{g}, \{\cdot,\cdot,\cdot\},\epsilon,\alpha,\beta)$;
\item The following equalities hold for all $x,y,z\in\mathcal{H}(\mathfrak{g})$:
\begin{eqnarray}
& \rho(\alpha\beta(x),\alpha\beta(y))\mathfrak l(z)=
\mathfrak l(\{\beta(x),\beta(y),z\})\beta_V+\epsilon(x+y,z)\mathfrak l(\beta(z))\rho(\alpha(x),\alpha(y)),\label{repr-BiHom-colr-Poisson1}\\
& \rho(\alpha\beta(x),\alpha\beta(y))\mathfrak r(z)=\epsilon(x+y,z)\mathfrak r(\beta(z))\rho(\beta(x),\beta(y))+\mathfrak r(\{\alpha(x),\alpha(y),z\})\beta_V,
\label{repr-BiHom-colr-Poisson2}\\
&\rho(\alpha\beta(x),\mu(\beta(y),\beta(z)))\alpha_V^2\beta_V=\epsilon(x+y,z)\mathfrak r(\alpha\beta(z))\rho(x,\beta^2(y))\alpha_V\beta_V\label{repr-BiHom-colr-Poisson3}\\
&+\epsilon(x,y)\mathfrak l(\alpha\beta(y))\rho(\alpha(x),\beta(z))\alpha_V^2.\nonumber
\end{eqnarray}
\end{enumerate}
\end{defi}
\begin{rmk}
If $\alpha$ and $\beta_V$ are bijective, then the condition \eqref{repr-BiHom-colr-Poisson3} is equivalent to
\begin{eqnarray}
&\rho(\alpha\beta(x),\mu(\alpha^{-1}\beta(y),\alpha^{-1}\beta(z)))\alpha_V^2=\epsilon(x+y,z)\mathfrak r(\beta(z))\rho(x,\alpha^{-1}\beta^2(y))\alpha_V\label{equiv-repr-BiHom-colr-Poisson3}\\
&+\epsilon(x,y)\mathfrak l(\beta(y))\rho(\alpha(x),\alpha^{-1}\beta(z))\alpha_V^2\beta_V^{-1}.\nonumber
\end{eqnarray}
\end{rmk}

\begin{thm}
Let $(\mathfrak{g},\{\cdot,\cdot,\cdot\},\mu,\epsilon,\alpha,\beta)$ be a noncommutative $3$-BiHom-Poisson color algebra where $\alpha$ is bijective, $V$ is a $\Gamma$-graded vector space and $(\alpha_V,\beta_V)\in End(V)\times GL(V)$. Let $\rho:\mathfrak{g}\times\mathfrak{g}\rightarrow End(V)$ be an even $\epsilon$-skew-symmetric bilinear map and $\mathfrak l,\mathfrak r:\mathfrak{g}\rightarrow End(V)$ be two even linear maps. Then, the tuple $(V,\rho,\mathfrak l,\mathfrak r,\alpha_V,\beta_V)$ is a representation of
$(\mathfrak{g},\{\cdot,\cdot,\cdot\},\mu,\epsilon,\alpha,\beta)$ if and only if, the direct sum $\mathfrak{g}\oplus V$  turns into a noncommutative $3$-BiHom-Poisson color algebra by defining the linear maps defined by
\eqref{alpha-direct-sum}-\eqref{beta-direct-sum} and two multiplications defined by \eqref{mu-direct-sum}, \eqref{crochet-direct-sum}.
\end{thm}
\begin{proof}
By Propositions \ref{3-BiHom-ass-direct-sum} and \ref{struc-3-BiHom-semidirect}, we show that $(\mathfrak{g}\oplus V,\mu_{\mathfrak{g}\oplus V},\epsilon,\alpha\oplus\alpha_V,\beta\oplus\beta_V)$ is a BiHom-associative color algebra and $(\mathfrak{g}\oplus V,\{\cdot,\cdot,\cdot\}_{\mathfrak{g}\oplus V},\epsilon,\alpha\oplus\alpha_V,\beta\oplus\beta_V)$ is a $3$-BiHom-Lie color algebra.

Now, we must demonstrate that the Hom-Leibniz identity \eqref{BiHom-Leibniz-role} is satisfied on $\mathfrak{g}\oplus V$. For $x,y,z,t\in\mathcal{H}(\mathfrak{g})$ and $u,v,w,\theta\in\mathcal{H}(V)$, 
\begin{align*}
& \{(\alpha\oplus\alpha_V)(\beta\oplus\beta_V)(x+u),(\alpha\oplus\alpha_V)(\beta\oplus\beta_V)(y+v),\mu_{\mathfrak{g}\oplus V}(z+w,t+\theta)\}_{\mathfrak g\oplus V}\\
& =\{\alpha\beta(x)+\alpha_V\beta_V(u),\alpha\beta(y)+\alpha_V\beta_V(v),\mu(z,t)+\mathfrak l(z)\theta+\mathfrak r(t)w\}_{\mathfrak{g}\oplus V}\\
& =\{\alpha\beta(x),\alpha\beta(y),\mu(z,t)\}+\rho(\alpha\beta(x),\alpha\beta(y))\big(\mathfrak l(z)\theta+\mathfrak r(t)w\big)\\
&\quad -\epsilon(v,z+t)\rho(\alpha\beta(x),\mu(\alpha^{-1}\beta(z),\alpha^{-1}\beta(t)))\alpha_V^2(v)\\
&\quad +\epsilon(u,y+z+t)\rho(\alpha\beta(y),\mu(\alpha^{-1}\beta(z),\alpha^{-1}\beta(t)))\alpha_V^2(u). \\
&\mu_{\mathfrak{g}\oplus V}\Big(\{\beta\oplus\beta_V(x+u),\beta\oplus\beta_V(y+v),z+w\}_{\mathfrak{g}\oplus V},\beta\oplus\beta_V(t+\theta)\Big)\\
&=\mu_{\mathfrak{g}\oplus V}\Big(\{\beta(x),\beta(y),z\}+\rho(\beta(x),\beta(y))w-\rho(x,\alpha^{-1}\beta(z))\alpha_V(v)\\
&\hspace{5cm} +\rho(y,\alpha^{-1}\beta(z))\alpha_V(u),\beta(t)+\beta_V(\theta)\Big)\\
& =\mu(\{\beta(x),\beta(y),z\},\beta(t))+\mathfrak l(\{\beta(x),\beta(y),z\})\beta_V(\theta)\\
&\quad +\epsilon(t,x+y+w)\mathfrak r(\beta(t))\Big(\rho(\beta(x),\beta(y))w-\epsilon(z,v)\rho(x,\alpha^{-1}\beta(z))\alpha_V(v)\\ &\hspace{6cm} +\epsilon(u,y+z)\rho(y,\alpha^{-1}\beta(z))\alpha_V(u)\Big). \\
&  \mu_{\mathfrak{g}\oplus V}\Big(\beta\oplus\beta_V(z+w),
\{\alpha\oplus\alpha_V(x+u), \alpha\oplus\alpha V(y+v), t+\theta\} {\mathfrak{g}\oplus V}\Big)\\
&=  \mu(\beta(z),\{\alpha(x),\alpha(y),t\})+\epsilon(t,w)\mathfrak r(\{\alpha(x),\alpha(y),t\})\beta_V(w)\\
&\quad +\mathfrak l(\beta(z))\Big(\rho(\alpha(x),\alpha(y))\theta-\epsilon(t,v)\rho(\alpha(x),\alpha^{-1}\beta(t))\alpha_V^2\beta_V^{-1}(v)\\
&\quad +\epsilon(u,y+t)\rho(\alpha(y),\alpha^{-1}\beta(t))\alpha_V^2\beta_V^{-1}(u)\Big).
\end{align*}
By the fact that $(\mathfrak{g},\{\cdot,\cdot,\cdot\},\mu,\epsilon,\alpha,\beta)$ is a noncommutative $3$-BiHom-Poisson color algebra, the BiHom-Leibniz identity \eqref{BiHom-Leibniz-role} is satisfied by $\mu_{\mathfrak g\oplus V}$ and $\{\cdot,\cdot,\cdot\}_{\mathfrak g\oplus V}$, if and only if
\begin{align*}
\rho(\alpha\beta(x),\alpha\beta(y))\mathfrak l(z)\theta&=\mathfrak l(\{\beta(x),\beta(y),z\})\beta_V(\theta) +\varepsilon(x+y,z)l(\beta(z))\rho(\alpha(x),\alpha(y))\theta,
\end{align*}
for any $x,y,z\in\mathcal{H}(\mathfrak g)$ and $\theta\in\mathcal{H}(V)$, which implies that the condition \eqref{repr-BiHom-colr-Poisson1} is satisfied. 

Similarly, we obtain, for any $x,y,z,t\in\mathcal{H}(\mathfrak g)$ and $v,w\in\mathcal{H}(V)$,
\begin{align*}
\rho(\alpha\beta(x),\alpha\beta(y))\mathfrak r(t)w &= \varepsilon(t,x+y)\mathfrak r(\beta(t))\rho(\beta(x),\beta(y))w+ \mathfrak r(\{\alpha(x),\alpha(y),t\})\beta_V(w),
\\
\rho(\alpha\beta(x),\mu(\alpha^{-1}\beta(z),\alpha^{-1}\beta(t)))\alpha_V^2(v)&=\varepsilon(t,x+z) \mathfrak r(\beta(t))\rho(x,\alpha^{-1}\beta(z))\alpha_V(v)\\&\quad+\varepsilon(x,z)\mathfrak l(\beta(z))\rho(\alpha(x),\alpha^{-1}\beta(t))\alpha_V^2\beta_V^{-1}(v),
\end{align*}
which gives the conditions \eqref{repr-BiHom-colr-Poisson2} and \eqref{repr-BiHom-colr-Poisson3} respectively. Thus, $(V,\rho,\mathfrak l,\mathfrak r,\alpha_V,\beta_V)$ is a representation of $(\mathfrak{g},\{\cdot,\cdot,\cdot\},\mu,\epsilon,\alpha,\beta)$. 
\end{proof}
\subsection{Kupershmidt operators on noncommutative $3$-BiHom-Poisson color algebras}
\begin{defi}
Let $(V,\rho,\mathfrak l,\mathfrak r,\alpha_V,\beta_V)$ be a representation of the noncommutative $3$-BiHom-Poisson color algebra $(\mathfrak{g},\{\cdot,\cdot,\cdot\},\mu,\epsilon,\alpha,\beta)$. An even linear map $T:V\to\mathfrak{g}$ is called {\it \bf{Kupershmidt operator}} of $\mathfrak{g}$ with respect to $(V,\rho,\mathfrak l,\mathfrak r,\alpha_V,\beta_V)$ if
\begin{enumerate}[label=\upshape{\arabic*)}, ref=\upshape{\arabic*}, labelindent=5pt, leftmargin=*]
\item $T$ is a Kupershmidt operator on the BiHom-associative color algebra $(\mathfrak{g},\mu,\epsilon,\alpha,\beta)$ with respect to the bimodule $(V,\mathfrak l,\mathfrak r,\alpha_V,\beta_V)$;
\item $T$ is a Kupershmidt on the $3$-BiHom-Lie color algebra $(\mathfrak{g},\{\cdot,\cdot,\cdot\},\epsilon,\alpha,\beta)$ with respect to the representation $(V,\rho,\alpha_V,\beta_V)$.
\end{enumerate}
If $\mathcal{R}$ is a Kupershmidh operator of $\mathfrak{g}$ with respect to representation 
$(V,ad,\mathfrak L,\mathfrak R,\alpha,\beta),$ 
then $\mathcal{R}$ is called a Rota-Baxter operator of weight $0$ on $\mathfrak g$, that is
\begin{enumerate}[label=\upshape{\arabic*)}, ref=\upshape{\arabic*}, labelindent=5pt, leftmargin=*]
\item $\mathcal{R}$ is a Rota-Baxter operator of weight $0$ on BiHom-associative color algebra 
$(\mathfrak g,\mu,\epsilon,\alpha,\beta)$;
\item $\mathcal{R}$ is a Rota-Baxter operator of weight $0$ on $3$-BiHom-Lie color algebra 
$(\mathfrak g,\{\cdot,\cdot,\cdot\},\epsilon,\alpha,\beta).$
\end{enumerate}
\end{defi}
\begin{thm}
Let $\mathcal{R}:\mathfrak{g}\to\mathfrak{g}$ be Rota-Baxter operator of weight $0$ on a noncommutative $3$-BiHom-Poisson color algebra $(\mathfrak{g},\{\cdot,\cdot,\cdot\},\mu,\epsilon,\alpha,\beta)$. Then, $(\mathfrak{g},\{\cdot,\cdot,\cdot\}_\mathcal{R},\mu_\mathcal{R},\epsilon,\alpha,\beta)$ is a noncommutative $3$-BiHom-Poisson color algebra where $\{\cdot,\cdot,\cdot\}_\mathcal{R}$ is defined by   \eqref{3-BiH-Lie-color-RB} and $\mu_\mathcal{R}$ given by   \eqref{RB-op-BiH-ass-color}.
\end{thm}
\begin{proof}
By Proposition \ref{ind-BiH-ass-color-RB}, the tuple $(\mathfrak{g},\mu_\mathcal{R},\epsilon,\alpha,\beta)$ is a BiHom-associative color algebra and Corollary \ref{ind-3-BiH-Lie-color-RB} gives that $(\mathfrak{g},\{\cdot,\cdot,\cdot\}_\mathcal{R},\epsilon,\alpha,\beta)$ is a $3$-BiHom-Lie color algebra. It remains to prove that the Leibniz identity \eqref{BiHom-Leibniz-role} is satisfied.
For $x,y,z,t\in\mathcal{H}(\mathfrak g)$,
\begin{align*}
&\{\alpha\beta(x),\alpha\beta(y),\mu_\mathcal{R}(z,t)\}_\mathcal{R}= \{\alpha\beta(x),\alpha\beta(y),\mu(\mathcal{R}(z),t)+\mu(z,\mathcal{R}(t))\}_\mathcal{R}\\&= \{\alpha\beta(x),\alpha\beta(y),\mu(\mathcal{R}(z),t)\}_\mathcal{R}+\{\alpha\beta(x),\alpha\beta(y),\mu(z,\mathcal{R}(t))\}_\mathcal{R}\\&=
\{\alpha\beta(\mathcal{R}x),\alpha\beta(\mathcal{R}y),\mu(\mathcal{R}(z),t)\} +\{\alpha\beta(\mathcal{R}x),\alpha\beta(y),\mathcal{R}\mu(\mathcal{R}(z),t)\}\\
&\quad+\{\alpha\beta(x),\alpha\beta(\mathcal{R}y),\mathcal{R}\mu(\mathcal{R}(z),t)\}+
\{\alpha\beta(\mathcal{R}x),\alpha\beta(\mathcal{R}y),\mu(z,\mathcal{R}t)\}\\&\quad+\{\alpha\beta(\mathcal{R}x),\alpha\beta(y),\mathcal{R}\mu(z,\mathcal{R}t)\}+\{\alpha\beta(x),\alpha\beta(\mathcal{R}y),\mathcal{R}\mu(z,\mathcal{R}t)\}\\&= \{\alpha\beta(\mathcal{R}x),\alpha\beta(\mathcal{R}y),\mu(\mathcal{R}(z),t)\}+ \{\alpha\beta(\mathcal{R}x),\alpha\beta(\mathcal{R}y),\mu(z,\mathcal{R}t)\}\\&\quad+\{\alpha\beta(\mathcal{R}x),\alpha\beta(y),\mu(\mathcal{R}z,\mathcal{R}t)\}+\{\alpha\beta(x),\alpha\beta(\mathcal{R}y),\mu(\mathcal{R}z,\mathcal{R}t)\}\\&=\mu(\{\beta(\mathcal{R}x),\beta(\mathcal{R}y),\mathcal{R}z\},\beta(t))+\varepsilon(z,x+y)\mu(\beta(\mathcal{R}z),\{\alpha(\mathcal{R}x),\alpha(\mathcal{R}y),t\})\\&\quad+\mu(\{\beta(\mathcal{R}x),\beta(\mathcal{R}y),z\},\beta(\mathcal{R}t))+\varepsilon(z,x+y)\mu(\beta(z),\{\alpha(\mathcal{R}x),\alpha(\mathcal{R}y),\mathcal{R}t\})\\&\quad+\mu(\{\beta(\mathcal{R}x),\beta(y),\mathcal{R}z\},\beta(\mathcal{R}t))+\varepsilon(z,x+y)\mu(\beta(\mathcal{R}z),\{\alpha(\mathcal{R}x),\alpha(y),\mathcal{R}t\})\\&\quad+\mu(\{\beta(x),\beta(\mathcal{R}y),\mathcal{R}z\},\beta(\mathcal{R}t))+\varepsilon(z,x+y)\mu(\beta(\mathcal{R}z),\{\alpha(x),\alpha(\mathcal{R}y),\mathcal{R}t\})\\&=\Big[\mu(\{\beta(\mathcal{R}x),\beta(\mathcal{R}y),\mathcal{R}z\},\beta(t))+\mu(\{\beta(\mathcal{R}x),\beta(\mathcal{R}y),z\},\beta(\mathcal{R}t))\\&\quad+\mu(\{\beta(\mathcal{R}x),\beta(y),\mathcal{R}z\},\beta(\mathcal{R}t))+\mu(\{\beta(x),\beta(\mathcal{R}y),\mathcal{R}z\},\beta(\mathcal{R}t))\Big]\\&\quad+\varepsilon(z,x+y)\Big[\mu(\beta(\mathcal{R}z),\{\alpha(\mathcal{R}x),\alpha(\mathcal{R}y),t\})+\mu(\beta(z),\{\alpha(\mathcal{R}x),\alpha(\mathcal{R}y),\mathcal{R}t\})\\&\quad+\mu(\beta(\mathcal{R}z),\{\alpha(\mathcal{R}x),\alpha(y),\mathcal{R}t\})+\mu(\beta(\mathcal{R}z),\{\alpha(x),\alpha(\mathcal{R}y),\mathcal{R}t\})\Big]\\&=\mu(\mathcal{R}\{\beta(x),\beta(y),z\}_\mathcal{R},\beta(t))+\mu(\{\beta(x),\beta(y),z\}_\mathcal{R},\mathcal{R}\beta(t))\\&\quad+\varepsilon(z,x+y)\Big(\mu(\mathcal{R}\beta(z),\{\alpha(x),\alpha(y),t\}_\mathcal{R})+\mu(\beta(z),\mathcal{R}\{\alpha(x),\alpha(y),t\}_\mathcal{R})\Big)\\&=\mu_\mathcal{R}(\{\beta(x),\beta(y),z\}_\mathcal{R},\beta(t))+\varepsilon(z,x+y)\mu_\mathcal{R}(\beta(z),\{\alpha(x),\alpha(y),t\}_\mathcal{R}).
\end{align*}
Then, the Leibniz identity \eqref{BiHom-Leibniz-role} is satisfied by $\mu_\mathcal{R}$ and $\{\cdot,\cdot,\cdot\}_\mathcal{R}$, which gives the Proof.
\end{proof}

\section{Noncommutative $3$-BiHom-pre-Poisson color algebras} 
\label{sec-noncom3bihomprepoisoncolalgs}

\begin{defi}
{\it \bf{$3$-BiHom-pre-Lie color algebras}} are defined as $5$-tuples  $(\mathfrak{g},\{\cdot,\cdot,\cdot\},\epsilon,\alpha,\beta)$, consisting of a linear super-space $\mathfrak{g}$, an even linear map
$\{\cdot,\cdot,\cdot\}:\mathfrak{g}\otimes \mathfrak{g}\otimes \mathfrak{g}\to \mathfrak{g}$ and two commuting linear maps $\alpha,\beta:\mathfrak{g}\to \mathfrak{g}$, such that the following identities hold for $x,y,z, x_i\in\mathcal{H}( \mathfrak{g}), 1\leq i\leq 5$ 
\begin{align}
&\alpha\{x,y,z\}=\{\alpha(x),\alpha(y),\alpha(z)\},\ \beta\{x,y,z\}=\{\beta(x),\beta(y),\beta(z)\},\label{3-BiHom-pre-Lie-color1}  \\
&\{\beta(x),\beta(y),\alpha(z)\}=-\epsilon(x,y)\{\beta(y),\beta(x),\alpha(z)\}, \label{3-BiHom-pre-Lie-color2}\\
\nonumber  &\{\alpha\beta(x_4),\alpha\beta(x_5),\{x_1,x_2,x_3\}\} =
\{[\beta(x_4),\beta(x_5),x_1]^C,\beta(x_2),\beta(x_3)\}   \\
&+\epsilon(x_1,x_4+x_5)\{\beta(x_1),[\beta(x_4),\beta(x_5),x_2]^C,\beta(x_3)\} \label{3-BiHom-pre-Lie-color3}\\
\nonumber &+\epsilon(x_1+x_2,x_4+x_5)\{\beta(x_1),\beta(x_2),\{\alpha(x_4),\alpha(x_5),x_3\}\},\\
\nonumber
&\{[\beta(x_4),\beta(x_5),x_1]^C,\beta(x_2),\beta(x_3)\} =
\epsilon(x_4,x_1+x_5)\{\alpha\beta(x_5),\beta(x_1),\{ \alpha(x_4),x_2, x_3\}\} \\
&+\epsilon(x_1,x_4+x_5)\{\beta(x_1),\alpha\beta(x_4),\{ \alpha(x_5),x_2,x_3\}\} \label{3-BiHom-pre-Lie-color4}\\
\nonumber &+\{\alpha\beta(x_4),\alpha\beta(x_5),\{ x_1,x_2,x_3\}\},
\end{align}
where the $3$-BiHom-$\epsilon$-commutator $[\cdot,\cdot,\cdot]^C$ is defined by
\begin{equation}\label{eq:3cc}
[x,y,z]^C=\{x,y,z\}-\epsilon(y,z)\{x,\beta\alpha^{-1}(z),\alpha\beta^{-1}(y)\}+\epsilon(x,y+z)\{y,\beta\alpha^{-1}(z),\alpha\beta^{-1}(x)\},
\end{equation}
for all  $x,y,z\in \mathcal{H}(\mathfrak{g})$.
\end{defi}
\begin{rmk}
If $\alpha=\beta=Id_\mathfrak{g}$, we recover the $3$-pre-Lie color algebra structures. If $\alpha=\beta$ and $\alpha$ is surjective, we recover the $3$-Hom-pre-Lie color algebra structures.
\end{rmk}
\begin{pdef}\label{pro:subadj}
Let $( \mathfrak{g},\{\cdot,\cdot,\cdot\},\epsilon,\alpha,\beta)$ be a $3$-BiHom-pre-Lie color algebra. Then, $( \mathfrak{g},[\cdot,\cdot,\cdot]^C,\epsilon,\alpha,\beta)$, where $[\cdot,\cdot,\cdot]^C$ is given by  \eqref{eq:3cc}
is a multiplicative $3$-BiHom-Lie color algebra called the sub-adjacent $3$-BiHom-Lie color algebra of $( \mathfrak{g},\{\cdot,\cdot,\cdot\},\epsilon,\alpha,\beta)$, and denoted by $ \mathfrak{g}^{c}$ and $( \mathfrak{g},\{\cdot,\cdot,\cdot\},\epsilon,\alpha,\beta)$ is called a compatible
$3$-BiHom-pre-Lie color algebra of the $3$-BiHom-Lie color algebra $ \mathfrak{g}^{c}$.
\end{pdef}
\begin{proof}
The multiplicativity of $[\cdot,\cdot,\cdot]^C$ is deduced directly from condition \eqref{3-BiHom-pre-Lie-color1}.
By condition \eqref{3-BiHom-pre-Lie-color1}, we obtain for $x_i\in\mathcal{H}(\mathfrak{g}),\;1\leq i\leq3$,
\begin{align*}
[\beta(x_2),\beta(x_1),\alpha(x_3)]^C&=\{\beta(x_2),\beta(x_1),\alpha(x_3)\}-\epsilon(x_1,x_3)\{\beta(x_2),\beta(x_3),\alpha(x_1)\}\\&\quad +\epsilon(x_2,x_1+x_3)\{\beta(x_1),\beta(x_3),\alpha(x_2)\}\\&=-\epsilon(x_1,x_2)\Big(\{\beta(x_1),\beta(x_2),\alpha(x_3)\}-\epsilon(x_2,x_3)\{\beta(x_1),\beta(x_3),\alpha(x_2)\}\\&\quad+\epsilon(x_1,x_2+x_3)\{\beta(x_2),\beta(x_3),\alpha(x_1)\}\Big)\\&= -\epsilon(x_1,x_2)[\beta(x_1),\beta(x_2),\alpha(x_3)]^C.
\end{align*}
Similarly,
$[\beta(x_1),\beta(x_2),\alpha(x_3)]^C= -\epsilon(x_2,x_3)[\beta(x_1),\beta(x_3),\alpha(x_2)]^C.$
Then, the conditions \eqref{Bihom-epsilon-symm1} and \eqref{Bihom-epsilon-symm2} are satisfied by $[\cdot,\cdot,\cdot]^C$.

It remains to show that \eqref{epsilon-BiHom-Nambu} is satisfied by $[\cdot,\cdot,\cdot]^C$.
For $x_i\in\mathcal{H}(\mathfrak{g}),\;1\leq i\leq3$,
\begin{align*}
M&=\epsilon(x_1+x_2+x_3,x_4+x_5)[\beta^2(x_4),\beta^2(x_5),[\beta(x_1),\beta(x_2),\alpha(x_3)]^{C}]^C\\
&\quad-\epsilon(x_1+x_2,x_3+x_5)\epsilon(x_4,x_5)[\beta^2(x_3),\beta^2(x_5),[\beta(x_1),\beta(x_2),\alpha(x_4)]^C]^C\\
&\quad+\epsilon(x_1+x_2,x_3+x_4)[\beta^{2}(x_3),\beta^{2}(x_4),[\beta(x_1),\beta(x_2),\alpha(x_5)]^{C}]^C\\
&=\epsilon(x_1+x_2+x_3,x_4+x_5)\Big(\{\beta^2(x_4),\beta^2(x_5),[\beta(x_1),\beta(x_2),\alpha(x_3)]^{C}\}\\
&\quad-\epsilon(x_1+x_2+x_3,x_5)\{\beta^{2}(x_4),\beta\alpha^{-1}[\beta(x_1),\beta(x_2),\alpha(x_3)]^C,\alpha\beta(x_5)\}\\
&\quad+\epsilon(x_4,x_1+x_2+x_3+x_5)\{\beta^{2}(x_5),\beta\alpha^{-1}[\beta(x_1),\beta(x_2),\alpha(x_3)]^C,\alpha\beta(x_4)\}\Big)\\
&\quad-\epsilon(x_1+x_2,x_3+x_5)\epsilon(x_4,x_5)\Big(\{\beta^{2}(x_3),\beta^{2}(x_5),[\beta(x_1),\beta(x_2),\alpha(x_4)]^C\}\\
&\quad-\epsilon(x_5,x_1+x_2+x_4)\{\beta^{2}(x_3),\beta^{3}\alpha^{-1}[\beta(x_1),\beta(x_2),\alpha(x_4)]^C,\alpha\beta(x_5)\}\\
&\quad+\epsilon(x_3,x_1+x_2+x_4+x_5)\{\beta^{2}(x_5),\beta\alpha^{-1}[\beta(x_1),\beta(x_2),\alpha(x_4)]^{C},\alpha\beta(x_3)\}\Big)\\
&\quad+\epsilon(x_1+x_2,x_3+x_4)\Big(\{\beta^{2}(x_3),\beta^{2}(x_4),[\beta(x_1),\beta(x_2),\alpha(x_5)]^{C}\}\\
&\quad-\epsilon(x_4,x_1+x_2+x_5)\{\beta^{2}(x_3),\alpha^{-1}\beta[\beta(x_1),\beta(x_2),\alpha(x_5)]^{C},\alpha\beta(x_4)\}\\
&\quad+\epsilon(x_3,x_1+x_2+x_4+x_5)\{\beta^{2}(x_4),\beta\alpha^{-1}[\beta(x_1),\beta(x_2),\alpha(x_5)]^{C},\alpha\beta(x_3)\}\Big)\\
&=\epsilon(x_1+x_2+x_3,x_4+x_5)\Big(\{\beta^2(x_4),\beta^2(x_5),
\{\beta(x_1),\beta(x_2),\alpha(x_3)\}\}\\
&\quad-\epsilon(x_2,x_3)(-1)^{|x_2||x_3|}\{\beta^2(x_4),\beta^2(x_5),
\{\beta(x_1),\beta(x_3),\alpha(x_2)\}\}\\
&\quad+\epsilon(x_1,x_2+x_3)
\{\beta^2(x_4),\beta^2(x_5),\{\beta(x_2),\beta(x_3),\alpha(x_1)\}\}\\
&\quad-\epsilon(x_5,x_1+x_2+x_3)\{\beta^{2}(x_4),\beta\alpha^{-1}[\beta(x_1),\beta(x_2),\alpha(x_3)]^C,\alpha\beta(x_5)\}\\
&\quad+\epsilon(x_4,x_1+x_2+x_3+x_5)\{\beta^{2}(x_5),\beta\alpha^{-1}[\beta(x_1),\beta(x_2),\alpha(x_3)]^C,\alpha\beta(x_4)\}\Big)\\
&\quad-\epsilon(x_1+x_2+x_4,x_3+x_5)\epsilon(x_4,x_5)\Big(\{\beta^{2}(x_3),\beta^{2}(x_5),[\beta(x_1),\beta(x_2),\alpha(x_4)]^C\}\\
&\quad-\epsilon(x_1+x_2+x_4,x_5)(\{\beta^{2}(x_3),\beta^{3}\alpha^{-1}
\{\beta(x_1),\beta(x_2),\alpha(x_4)\},\alpha\beta(x_5)\}\\
&\quad-\epsilon(x_2,x_4)\{\beta^{2}(x_3),\beta^{3}\alpha^{-1}
\{\beta(x_1),\beta(x_4),\alpha(x_2)\},\alpha\beta(x_5)\}\\
&\quad+\epsilon(x_1,x_2+x_4)\{\beta^{2}(x_3),\beta^{3}\alpha^{-1}
\{\beta(x_4),\beta(x_2),\alpha(x_1)\},\alpha\beta(x_5)\})\\
&\quad+\epsilon(x_3,x_1+x_2+x_4+x_5)\{\beta^{2}(x_5),\beta\alpha^{-1}[\beta(x_1),\beta(x_2),\alpha(x_4)]^{C},\alpha\beta(x_3)\}\Big)\\
&\quad+\epsilon(x_1+x_2,x_3+x_4)\Big(\{\beta^{2}(x_3),\beta^{2}(x_4),
\{\beta(x_1),\beta(x_2),\alpha(x_5)\}\}\\
&\quad-\epsilon(x_2,x_5)\{\beta^{2}(x_3),\beta^{2}(x_4),
\{\beta(x_1),\beta(x_5),\alpha(x_2)\}\}\\
&\quad+\epsilon(x_1,x_2+x_5)\{\beta^{2}(x_3),\beta^{2}(x_4),
\{\beta(x_2),\beta(x_5),\alpha(x_1)\}\}\\
&\quad-\epsilon(x_4,x_1+x_2+x_5)\{\beta^{2}(x_3),\alpha^{-1}\beta[\beta(x_1),\beta(x_2),\alpha(x_5)]^{C},\alpha\beta(x_4)\}\\
&\quad+\epsilon(x_3,x_1+x_2+x_4+x_5)\{\beta^{2}(x_4),\beta\alpha^{-1}[\beta(x_1),\beta(x_2),\alpha(x_5)]^{C},\alpha\beta(x_3)\}\Big).
\end{align*}
Applying \eqref{3-BiHom-pre-Lie-color2} and \eqref{3-BiHom-pre-Lie-color3}, we obtain 
\begin{align*}
&M=\{\beta^2(x_1),\beta^2(x_2),\{\beta(x_3),\beta(x_4),\alpha(x_5)\}\}-\epsilon(x_4,x_5)\{\beta^2(x_1),\beta^2(x_2),\{\beta(x_3),\beta(x_5),\alpha(x_4)\}\}\\&+\epsilon(x_3,x_4+x_5)\{\beta^2(x_1),\beta^2(x_2),\{\beta(x_4),\beta(x_5),\alpha(x_3)\}\}\\
&-\epsilon(x_2,x_3+x_4+x_5)\{\beta^{2}(x_1),\beta\alpha^{-1}[\beta(x_3),\beta(x_4),\alpha(x_5)]^C,\alpha\beta(x_2)\}\\
&+\epsilon(x_1,x_2+x_3+x_4+x_5)\{\beta^{2}(x_2),\beta\alpha^{-1}[\beta(x_3),\beta(x_4),\alpha(x_5)]^C,\alpha\beta(x_1)\}\\&
=\{\beta^2(x_1),\beta^2(x_2),[\beta(x_3),\beta(x_4),\alpha(x_5)]^{C}\}\\
&-\epsilon(x_2,x_3+x_4+x_5)\{\beta^{2}(x_1),\beta\alpha^{-1}[\beta(x_3),\beta(x_4),\alpha(x_5)]^C,\alpha\beta(x_2)\}\\
&+\epsilon(x_1,x_2+x_3+x_4+x_5)\{\beta^{2}(x_2),\beta\alpha^{-1}[\beta(x_3),\beta(x_4),\alpha(x_5)]^C,\alpha\beta(x_1)\}\\&
=[\beta^2(x_1),\beta^2(x_2),[\beta(x_3),\beta(x_4),\alpha(x_5)]^{C}]^{C}.
\end{align*}
Then  $(\mathfrak{g},[\cdot,\cdot,\cdot]^C,\epsilon,\alpha,\beta)$
is a $3$-BiHom-Lie color algebra.
\end{proof}
\begin{defi}
Let $( \mathfrak{g}_1,\{\cdot,\cdot,\cdot\}_1,\epsilon,\alpha_1,\beta_1)$ and $( \mathfrak{g}_2,\{\cdot,\cdot,\cdot\}_2,\epsilon,\alpha_2,\beta_2)$ be two $3$-BiHom-Lie color algebras.
\begin{enumerate}[label=\upshape{\arabic*.}, ref=\upshape{\arabic*}, labelindent=5pt, leftmargin=*]
\item An even linear map $f:\mathfrak{g}_1\to\mathfrak{g}_2$ is called a weak homomorphism of $3$-BiHom-Lie color algebras if, for all $x,y,z\in\mathcal{H}(\mathfrak{g}_1)$,
$$f(\{x,y,z\}_1)=\{f(x),f(y),f(z)\}_2.$$
\item A weak homomorphism $f:\mathfrak{g}_1\to\mathfrak{g}_2$ is called a homomorphism of $3$-BiHom-Lie color algebras if $f$ also satisfies
$\alpha_2\circ f=f\circ\alpha_1\;\;\;\text{and}\;\;\;\beta_2\circ f=f\circ\beta_1.$
\end{enumerate}
\end{defi}
\begin{thm}\label{3-BiH-pre-Lie-twis-thm}
Let $(\mathfrak{g},\{\cdot,\cdot,\cdot\},\epsilon)$ be a $3$-pre-Lie color algebra and $\alpha,\beta$ two commuting algebra automorphisms of $\mathfrak g$. Then, the $5$-tuple $(\mathfrak{g},\{\cdot,\cdot,\cdot\}_{\alpha,\beta},\epsilon,\alpha,\beta)$ is a $3$-BiHom-pre-Lie color algebra, where the even trilinear map $\{\cdot,\cdot,\cdot\}_{\alpha,\beta}$ is defined for all $x,y,z\in\mathcal{H}(\mathfrak{g})$ by
\begin{equation}\label{3-BiH-pre-Lie-twis-cro}
\{x,y,z\}_{\alpha,\beta}=\{\alpha(x),\alpha(y),\beta(z)\}.
\end{equation}
\end{thm}
\begin{proof}
Using \eqref{eq:3cc} and Theorem \ref{thm-twist-assoc-color}, we deduce that the even trilinear map $[\cdot,\cdot,\cdot]^C_{\alpha,\beta}$ define a $3$-BiHom-Lie color algebra structure on $\mathfrak{g}$, where, for all $x,y,z\in\mathcal{H}(\mathfrak{g})$,
\begin{align*} [x,y,z]^C_{\alpha,\beta}&=\{x,y,z\}_{\alpha,\beta}
-\epsilon(y,z)\{x,\beta\alpha^{-1}(z),\alpha\beta^{-1}(y)\}_{\alpha,\beta}\\&+\epsilon(x,y+z)\{y,\beta\alpha^{-1}(z),\alpha\beta^{-1}(x)\}_{\alpha,\beta}\\&=[\alpha(x),\alpha(y),\beta(z)]^C.
\end{align*}

For $x_i\in\mathcal{H}(\mathfrak{g}),\;1\leq i\leq3$,
\begin{align*}
\{\beta(x_1),\beta(x_2),\alpha(x_3)\}_{\alpha,\beta}&=\{\alpha\beta(x_1),\alpha\beta(x_2),\alpha\beta(x_3)\}\\&=-\epsilon(x_1,x_2)\{\alpha\beta(x_2),\alpha\beta(x_1),\alpha\beta(x_3)\}\\&=-\epsilon(x_1,x_2)\{\beta(x_2),\beta(x_1),\alpha(x_3)\}_{\alpha,\beta}.
\end{align*}
Now, we need to show that $\{\cdot,\cdot,\cdot\}_{\alpha,\beta}$ satisfies conditions \eqref{3-BiHom-pre-Lie-color3} and \eqref{3-BiHom-pre-Lie-color4}. Let $x_i\in\mathcal{H}(\mathfrak{g}),\;1\leq i\leq5$. Then, using the fact that $(\mathfrak g,\{\cdot,\cdot,\cdot\},\varepsilon)$ is a $3$-pre-Lie color algebra, we obtain
\begin{align*}
&\{\alpha\beta(x_4),\alpha\beta(x_5),\{x_1,x_2,x_3\}_{\alpha,\beta}\}_{\alpha,\beta}=\{\alpha^{2}\beta(x_4),\alpha^{2}\beta(x_5),\{\alpha\beta(x_1),\alpha\beta(x_2),\beta^{2}(x_3)\}\}\\&=\{[\alpha^{2}\beta(x_4),\alpha^{2}\beta(x_5),\alpha\beta(x_1)]^C,\alpha\beta(x_2),\beta^{2}(x_3)\}\\&\quad +\epsilon(x_1,x_4+x_5)\{\alpha\beta(x_1),[\alpha^{2}\beta(x_4),\alpha^{2}\beta(x_5),\alpha\beta(x_2)]^C,\beta^{2}(x_3)\}\\&\quad+\epsilon(x_1+x_2,x_4+x_5)\{\alpha\beta(x_1),\alpha\beta(x_2),\{\alpha^{2}\beta(x_4),\alpha^{2}\beta(x_5),\beta^{2}(x_3)\}\}\\
&=\{[\beta(x_4),\beta(x_5),x_1]_{\alpha,\beta}^C,\beta(x_2),\beta(x_3)\}_{\alpha,\beta}\\
&\quad+\epsilon(x_1,x_4+x_5)\{\beta(x_1),[\beta(x_4),\beta(x_5),x_2]_{\alpha,\beta}^C,\beta(x_3)\}_{\alpha,\beta}\\
&\quad+\epsilon(x_1+x_2,x_4+x_5)\{\beta(x_1),\beta(x_2),\{\alpha(x_4),\alpha(x_5),x_3\}_{\alpha,\beta}\}_{\alpha,\beta},
\\
&\{[\beta(x_4),\beta(x_5),x_1]_{\alpha,\beta}^C,\beta(x_2),\beta(x_3)\}_{\alpha,\beta} =
\{\alpha[\beta(x_4),\beta(x_5),x_1]_{\alpha,\beta}^C,\alpha\beta(x_2),\beta^2(x_3)\}\\
&=\{[\alpha^2\beta(x_4),\alpha^2\beta(x_5),\alpha\beta(x_1)]^C,
\alpha\beta(x_2),\beta^2(x_3)\}\\&=\epsilon(x_4,x_1+x_5)
\{\alpha^2\beta(x_5),\alpha\beta(x_1),\{ \alpha^2\beta(x_4),\alpha\beta(x_2), \beta^2(x_3)\}\} \\
&\quad+\epsilon(x_1,x_4+x_5)\{\alpha\beta(x_1),\alpha^2\beta(x_4),\{ \alpha^2\beta(x_5),\alpha\beta(x_2),\beta^2(x_3)\}\}\\
&\quad+\{\alpha^2\beta(x_4),\alpha^2\beta(x_5),
\{ \alpha\beta(x_1),\alpha\beta(x_2),\beta^2(x_3)\}\}\\
&=\epsilon(x_4,x_1+x_5)\{\alpha\beta(x_5),\beta(x_1),\{ \alpha(x_4),x_2, x_3\}_{\alpha,\beta}\}_{\alpha,\beta} \\
&\quad+\epsilon(x_1,x_4+x_5)\{\beta(x_1),\alpha\beta(x_4),\{ \alpha(x_5),x_2,x_3\}_{\alpha,\beta}\}_{\alpha,\beta} \\
&\quad+\{\alpha\beta(x_4),\alpha\beta(x_5),\{ x_1,x_2,x_3\}_{\alpha,\beta}\}_{\alpha,\beta},
\end{align*}
which implies that the conditions \eqref{3-BiHom-pre-Lie-color3} and \eqref{3-BiHom-pre-Lie-color4} are satisfied. Then $(\mathfrak{g},\{\cdot,\cdot,\cdot\}_{\alpha,\beta},\epsilon,\alpha,\beta)$ is a $3$-BiHom-pre-Lie color algebra.
\end{proof}
\begin{thm}\label{3-BiH-pre-Lie-idu-Kup}
Let $(\mathfrak{g},[\cdot,\cdot,\cdot],\varepsilon,\alpha,\beta)$ be a $3$-BiHom-Lie color algebra and $(V,\rho,\alpha_V,\beta_V)$ be a representation on $\mathfrak g$. Suppose that, the even linear map $T:V\to\mathfrak{g}$ is a Kupershmidt operator associated with $(V,\rho,\alpha_V,\beta_V)$. Then, there exists a $3$-BiHom-pre-Lie color algebra structure on $V$ given by
$$\{u,v,w\}_V=\rho(T(u),T(v))w,\;\forall u,v,w\in\mathcal{H}(V).$$
\end{thm}
\begin{proof}
For $u,v,w\in\mathcal{H}(V)$,
\begin{align*}
\alpha_V\{u,v,w\}_V&=\alpha_V\circ\rho(T(u),T(v))w=\rho(\alpha(T(u)),\alpha(T(v))\alpha_V(w)\\
& =\rho(T(\alpha_V(u)),T(\alpha_V(v)))\alpha_V(w)=\{\alpha_V(u),\alpha_V(v),\alpha_V(w)\}_V.
\end{align*}
Similarly, one can show that $\beta_V\{u,v,w\}_V=\{\beta_V(u),\beta_V(v),\beta_V(w)\}_V$,
and condition \eqref{3-BiHom-pre-Lie-color1} is therefore satisfied. Since 
\begin{align*}
\{\beta_V(u),\beta_V(v),\alpha(w)\}_V&=\rho(T(\beta_V(u)),T(\beta_V(v)))\alpha(w)\\&=-\varepsilon(u,v) \rho(T(\beta_V(v)),T(\beta_V(u)))\alpha(w)\\&=-\varepsilon(u,v)\{\beta_V(v),\beta_V(u),\alpha(w)\}_V,
\end{align*}
$\{\cdot,\cdot,\cdot\}_V$ satisfies condition \eqref{3-BiHom-pre-Lie-color2}.\\

Because $T$ is a Kupershmidt operator associated with $(V,\rho,\alpha_V,\beta_V)$, we have
\begin{align*}
T[u,v,w]_V^C&=T\Big(\rho(T(u),T(v))w-\varepsilon(v,w)\rho(T(u),T(\alpha_V^{-1}\beta_V(u)))\alpha_V\beta^{-1}_V(v)\\
&\quad\quad+\varepsilon(u,v+w)\rho(T(v),T((\alpha_V^{-1}\beta_V(w)))\alpha_V\beta_V^{-1}(u)\Big)\\&=[T(u),T(v),T(w)].
\end{align*}
For any $u,v,w,t,s\in\mathcal{H}(V)$, using\eqref{cond-O-op-3-BiH-Lie-color1}, we obtain 
\begin{align*}
\{\alpha_V\beta_V(u),\alpha_V\beta_V(v),\{w,s,t\}_v\}_V&=\rho\Big(T(\alpha_V\beta_V(u)),T(\alpha_V\beta_V(v))\Big)\rho\Big(T(w),T(s)\Big)(t)\\
\{[\beta_V(u),\beta_V(v),w]^C_V,\beta_V(s),\beta_V(t)\}_V&=\rho\Big(T([\beta_V(u),\beta_V(v),w]^C_V),T(\beta_V(s))\Big)\beta_V(t)\\&=\rho\Big([T(\beta_V(u)),T(\beta_V(v)),T(w)],T(\beta_V(s))\Big)\beta_V(t)\\&=\rho\Big([\beta(T(u)),\beta(T(v)),T(w)],\beta(T(s))\Big)\beta_V(t)\\ \{\beta_V(w),[\beta_V(u),\beta_V(v),s]^C_V,\beta_V(t)\}_V&=\rho\Big(T(\beta_V(w)),T([\beta_V(u),\beta_V(v),s]^C_V)\Big)\beta_V(t)\\&=\rho\Big(\beta(T(w)),[\beta(T(u)),\beta(T(v)),T(s)]\Big)\beta_V(t)\\
\{\beta_V(w),\beta_V(s),\{\alpha_V(u),\alpha_V(v),t\}_V\}_V&=\rho\Big(T(\beta_V(w)),T(\beta_V(s))\Big)\rho\Big(T(\alpha_V(u)),T(\alpha_V(v))\Big)(t)\\&=\rho\Big(\beta(T(w)),\beta(T(s))\Big)\rho\Big(\alpha(T(u)),\alpha(T(v))\Big)(t).
\end{align*}
Using the change  $T(u)=U,\;T(v)=V,T(w)=W$ and $T(s)=S$ and applying the fact that $(V,\rho,\alpha_V,\beta_V)$ is a representation of $(\mathfrak{g},[\cdot,\cdot,\cdot],\alpha,\beta)$, we can check that \eqref{3-BiHom-pre-Lie-color3} holds.
In the same manner, one can show that the condition \eqref{3-BiHom-pre-Lie-color4} is satisfied.
\end{proof}
\begin{cor}\label{ind-RB-3-BiH-pre-Lie-color}
Let $\mathcal{R}$ be a Rota-Baxter operator of weight zero on a $3$-BiHom-Lie color algebra $( \mathfrak{g},[\cdot,\cdot,\cdot],\epsilon,\alpha,\beta)$. Then $ \mathfrak{g}_\mathcal{R}=( \mathfrak{g},\{\cdot,\cdot,\cdot\}_\mathcal{R},\epsilon,\alpha,\beta)$ is a $3$-BiHom-pre-Lie color algebra, where $\{x,y,z\}_\mathcal{R}=[\mathcal{R}(x),\mathcal{R}(y),z]$
for all $x,y,z\in\mathcal{H}( \mathfrak{g})$,
\end{cor}

Now, we introduce the notion of BiHom-dendriform color algebras.
\begin{defi}\label{Def-BiHom-color-dend}
A BiHom-dendriform color algebra is a quintuple $(\mathfrak{g}, \prec, \succ,\epsilon,
\alpha,\beta)$ consisting of a $\Gamma$-graded vector space $\mathfrak{g}$ with the operations $\prec, \succ: \mathfrak{g}\otimes \mathfrak{g}\rightarrow
\mathfrak{g}$ and $\alpha,\beta: \mathfrak{g}\rightarrow \mathfrak{g}$ which are even linear maps satisfying, for all $x,y,z\in\mathcal{H}(\mathfrak{g})$,
\begin{align}
\alpha\circ\beta&=\beta\circ\alpha,\nonumber\\
\alpha(x \prec y)&=\alpha(x)\prec\alpha(y),\,\;\;\;
\beta(x \prec y)=\beta(x)\prec\beta(y),\nonumber\\
\alpha(x \succ y)&=\alpha(x)\succ\alpha(y),\;\;\;
\beta(x \succ y)=\beta(x)\succ\beta(y),\nonumber\\
(x \prec y)\prec \beta(z) &= \alpha(x)\prec (y \cdot z),\label{cond-BiHom-color-dend1}\\ (x
\succ y) \prec\beta(z) &= \alpha(x) \succ (y \prec z),\label{cond-BiHom-color-dend2}\\
\alpha(x)\succ (y \succ z ) &= ( x \cdot y) \succ \beta(z),\label{cond-BiHom-color-dend3}
\end{align}
where
\begin{eqnarray}\label{associative-dendriform}
x \cdot y = x \prec y + x \succ y,
\end{eqnarray}
\end{defi}
\begin{rmk}
\begin{enumerate}[label=\upshape{\arabic*.}, ref=\upshape{\arabic*}, labelindent=5pt, leftmargin=*]
\item
If $\alpha=\beta=Id$, we recover the dendriform color algebra structure.
\item If $\alpha=\beta$, we get the Hom-dendriform color algebra structure.
\end{enumerate}
\end{rmk}
\begin{pro}\label{twist-BiH-dend-pro}
Let $(\mathfrak{g}, \prec, \succ,\epsilon)$ be dendriform color algebra and $\alpha,\beta:\mathfrak{g}\to\mathfrak{g}$ be two even commuting algebra homomorphisms, that is
\begin{align*}
& \alpha(x \prec y)=\alpha(x)\prec\alpha(y),\quad
\beta(x \prec y)=\beta(x)\prec\beta(y)\\
& \alpha(x \succ y)=\alpha(x)\succ\alpha(y),\quad
\beta(x \succ y)=\beta(x)\succ\beta(y).
\end{align*}
Define the two binary operations $\prec_{\alpha,\beta},\ \succ_{\alpha,\beta}:\mathfrak{g}\times\mathfrak{g}\to\mathfrak{g}$,
for all $x,y\in\mathcal{H}(\mathfrak{g})$, by
\begin{equation}\label{twist-struc-BiH-dend}
x\prec_{\alpha,\beta}y=\alpha(x)\prec\beta(y)\;\;\;\text{and}\;\;\;x\succ_{\alpha,\beta}y=\alpha(x)\succ\beta(y).
\end{equation}
Then $(\mathfrak{g},\prec_{\alpha,\beta},\succ_{\alpha,\beta},\epsilon,\alpha,\beta)$ is a BiHom-dendriform color algebra.
\end{pro}
\begin{proof}
For $x,y,z\in\mathcal{H}(\mathfrak{g})$,
\begin{align*}
(x\prec_{\alpha,\beta}y)\prec_{\alpha,\beta}\beta(z)&=\alpha(\alpha(x)\prec\beta(y))\prec\beta^2(z)=\alpha^2(x)\prec\alpha\beta(y)\prec\beta^2(z)\\&=\alpha^2(x)\prec(\alpha\beta(y)\cdot\beta^2(z))\\&=\alpha(\alpha(x))\prec(\alpha\beta(y)\prec\beta^2(z)+\alpha\beta(y)\succ\beta^2(z)) \\&= \alpha(\alpha(x))\prec(\beta(y)\prec_{\alpha,\beta}\beta(z)+\beta(y)\succ_{\alpha,\beta}\beta(z))\\&= \alpha(x)\prec_{\alpha,\beta}(y\prec_{\alpha,\beta}z+y\succ_{\alpha,\beta}z)\\&=\alpha(x)\prec_{\alpha,\beta}(y\cdot_{\alpha,\beta}z),
\\
(x\succ_{\alpha,\beta} y) \prec_{\alpha,\beta}\beta(z)&= (\alpha^2(x)\succ\alpha\beta(y)) \prec\beta^2(z)= \alpha^2(x)\succ(\alpha\beta(y) \prec\beta^2(z))\\&= \alpha^2(x)\succ\beta(\alpha(y) \prec\beta(z))\\&= \alpha(x)\succ_{\alpha,\beta}(\alpha(y) \prec_{\alpha,\beta}\beta(z)),
\\
\alpha(x)\succ_{\alpha,\beta} (y \succ_{\alpha,\beta} z ) &= \alpha^2(x)\succ(\alpha\beta(y)\succ\beta^2(z))=(\alpha^2(x)\cdot\alpha\beta(y))\succ\beta^2(z))\\&= \alpha(\alpha(x)\cdot\beta(y))\succ\beta^2(z))\\&=(\alpha(x)\cdot\beta(y))\succ_{\alpha,\beta}\beta(z))\\&=(x\cdot_{\alpha,\beta} y)\succ_{\alpha,\beta}\beta(z)).
\end{align*}
Then, the conditions \eqref{cond-BiHom-color-dend1}-\eqref{cond-BiHom-color-dend3} are satisfied. The Proposition is proved.
\end{proof}
\begin{lem}\label{BiH-ass-from-BiH-den}
If $(\mathfrak{g}, \prec, \succ,\epsilon,\alpha,\beta)$ is a BiHom dendriform color algebra. Then, \\
$(\mathfrak{g},\mu:= \prec+\succ,\epsilon,\alpha,\beta)$ is a BiHom-associative color algebra.
\end{lem}
\begin{proof}
For any $x,y,z\in\mathcal{H}(\mathfrak{g})$, by using   \eqref{cond-BiHom-color-dend1}-\eqref{cond-BiHom-color-dend3}, we obtain
\begin{align*}
ass_{\alpha,\beta}(x,y,z)&=\mu(\alpha(x),\mu(y,z))-\mu(\mu(x,y),\beta(z))\\&=\alpha(x)\prec\mu(y,z)+\alpha(x)\succ\mu(y,z)-\mu(x,y)\prec\beta(z)-\mu(x,y)\succ\beta(z)\\&= \alpha(x)\prec (y\prec z)+\alpha(x)\prec (y\succ z)+\alpha(x)\succ(y\prec z)+\alpha(x)\succ(y\succ z)\\&-(x\prec y)\prec\beta(z)-(x\succ y)\prec\beta(z)-(x\prec y)\succ\beta(z)-(x\succ y)\succ\beta(z)\\&=(\alpha(x)\prec(y\cdot z)-(x\prec y)\prec\beta(z))+(\alpha(x)\succ(y\succ z)-(x\cdot y)\succ\beta(z))\\&+(\alpha(x)\succ(y\prec z)-(x\succ z)\prec\beta(z))=0,
\end{align*}
which gives the result.
\end{proof}
\begin{thm}\label{BiH-dedr-color-ind-Kup}
If $T$ is a Kupershmidt operator on BiHom-associative color algebra $(\mathfrak{g},\mu,\alpha,\beta)$ 
with respect to the bimodule $(V,\mathfrak{l},\mathfrak{r},\alpha)$, then there exists a BiHom-dendriform color algebra structure on $V$ with even bilinear maps $\prec_V,\succ_V$ defined for all $u,v\in\mathcal{H}(V)$ by
$$u\prec_Vv=\epsilon(u,v)\mathfrak{r}\big(T(v)\big)u,\;\;\;\;u\succ_Vv=\mathfrak{l}\big(T(u)\big)v.$$
\end{thm}
\begin{proof}
Let $u,v,w\in\mathcal{H}(V)$. By the fact that $T$ is a Kupershmidt operator on $(\mathfrak{g},\mu,\alpha,\beta)$, we have
\begin{align*}
&\alpha(u)\prec_V (v \cdot_V w)-(u \prec_V v)\prec_V \beta(w)\\
&=\epsilon(u,v+w)\epsilon(v,w)\alpha(u)\prec_V (v \prec_V w)+\epsilon(u,v+w)\alpha(u)\prec_V (v \succ_V w)\\
&\quad -\epsilon(u,v+w)(u \prec_V v)\prec_V \beta(w)\\&=\epsilon(u,v+w)\epsilon(v,w)\mathfrak{r}\Big(T(\mathfrak{r}(T(w)))v\Big)\alpha_V(u)+\mathfrak{r}\Big(T(\mathfrak{l}(T(v)))w\Big)\alpha_V(u)\\
&\quad-\epsilon(u,v+w)\mathfrak{r}(T(\beta_V(w)))\mathfrak{r}(T(v))u\\
&=\epsilon(u,v+w)\mathfrak{r}\Big(\epsilon(v,w)T(\mathfrak{r}(T(w)))v+T(\mathfrak{l}(T(v)))w\Big)\alpha_V(u)\\
&\quad-\epsilon(u,v+w)\mathfrak{r}(\mu(T(v),T(w)))\alpha_V(u)\\
&=\epsilon(u,v+w)\mathfrak{r}(\mu(T(v),T(w)))\alpha_V(u)-\epsilon(u,v+w)\mathfrak{r}(\mu(T(v),T(w)))\alpha_V(u) =0,
\\
&(u\succ_Vv)\prec_V\beta_V(w)=\varepsilon(u+v,w)\mathfrak{r}\big(T(\beta_V(w))\big)(u\succ_Vv)\\&= \varepsilon(u+v,w)\mathfrak{r}\big(T(\beta_V(w))\big)(\mathfrak{l}(T(u))v)= \varepsilon(u+v,w)\mathfrak{r}\big(\beta T((w))\big)(\mathfrak{l}(T(u))v)\\&=\varepsilon(v,w)\mathfrak{l}(\alpha(Tu))\mathfrak{r}(Tw)v=\mathfrak{l}(T\alpha_V(u))(v\prec_V w)= \alpha_V(u)\succ_V(v\prec_V w).
\end{align*}
Then, the conditions \eqref{cond-BiHom-color-dend1} and \eqref{cond-BiHom-color-dend2} are satisfied by $\prec_V$ and $\succ_V$ on $V$, and in a similar way, we show that, the condition \eqref{cond-BiHom-color-dend3} is satisfied.
\end{proof}
\begin{cor}\label{ind-RB-BiH-dend-color}
Let $\mathcal{R}$ be a Rota-Baxter operator of weight zero on a BiHom-associative color algebra $(\mathfrak{g},\mu,\epsilon,\alpha,\beta)$. Then, there exists a BiHom-dendriform color algebra structure on $\mathfrak{g}$ given by the two even bilinear maps $\prec_{\mathcal{R}}$ and $\succ_{\mathcal{R}}$ defined for all $x,y\in\mathcal{H}(\mathfrak{g})$ by
$$x\prec_{\mathcal{R}}y=\mu(\mathcal{R}(x),y),\;\;\;x\succ_{\mathcal{R}}y=\mu(x,\mathcal{R}(y)).$$
\end{cor}
\begin{defi}
Non-commutative $3$-BiHom-pre-Poisson color algebras 
$(\mathfrak{g},\{\cdot,\cdot,\cdot\},\prec,\succ,\epsilon,\alpha,\beta)$ consist of
a $3$-BiHom-pre-Lie color algebra $(\mathfrak{g},\{\cdot,\cdot,\cdot\},\epsilon,\alpha,\beta)$ and a BiHom-dendriform color algebra $(\mathfrak{g},\prec,\succ,\epsilon,\alpha,\beta)$ satisfying the following compatibility conditions:
\begin{align}
& \{\alpha\beta^2(x),\alpha\beta^2(y),z\prec \alpha(t)\}-\{\beta^2(x),\beta^2(y),z\}\prec \alpha\beta(t) \nonumber\\
&=\epsilon(z,x+y)\beta(z)\prec\Big(\{\alpha\beta(x),\alpha\beta(y),\alpha(t)\}
-\epsilon(y,t)\{\alpha\beta(x),\beta(t),\alpha^2(y)\}
\label{cond-3-BiH-pre-Poiss-color1}\\
&\hspace{5cm} +\epsilon(x,y+t)\{\alpha\beta(y),\beta(t),\alpha^2(x)\}\Big),\nonumber \\
& \{\alpha\beta(x),\alpha^2\beta(y),\alpha(z)\succ t\}-\epsilon(x+y,z)\alpha\beta(z)\succ\{\alpha(x),\alpha^2(y),t\}\nonumber\\
&=\Big(\{\beta(x),\alpha\beta(y),\alpha(z)\}-\epsilon(y,z)\{\beta(x),\beta(z),\alpha^2(y)\}
\label{cond-3-BiH-pre-Poiss-color2}\\
&\hspace{3cm}
+\epsilon(x,y+z)\{\alpha\beta(y),\beta(z),\alpha(x)\}\Big)\succ\beta(t),\nonumber\\
&\{\alpha\beta(x),\beta(y\prec z),\alpha^2\beta(t)\}+\{\alpha\beta(x),\beta(y\succ z),\alpha^2\beta(t)\}\label{cond-3-BiH-pre-Poiss-color3}\\
&=\epsilon(x,y)\alpha\beta(y)\succ\{\alpha(x),\beta(z),\alpha^2(t)\}
+\epsilon(z,t)\{\beta(x),\beta(y),\alpha\beta(t)\}\prec\alpha\beta(z). \nonumber
\end{align}

\end{defi}
\begin{rmk}
If $\alpha=\beta=Id$ (resp. $\alpha=\beta$ ), we recover the noncommutative $3$-pre-Poisson (resp. $3$-Hom-pre-Poisson) color algebras structures.
\end{rmk}
\begin{thm}
Let $(\mathfrak{g},\{\cdot,\cdot,\cdot\},\prec,\succ,\epsilon)$ be a noncommutative $3$-pre-Poisson color algebra and $\alpha,\;\beta$ be two even commuting algebra automorphisms of $\mathfrak{g}$. Then, \\
$(\mathfrak{g},\{\cdot,\cdot,\cdot\}_{\alpha,\beta},\prec_{\alpha,\beta},
\succ_{\alpha,\beta},\epsilon,\alpha,\beta)$ is a $3$-BiHom-pre-Poisson color algebra, where $\{\cdot,\cdot,\cdot\}_{\alpha,\beta},\;\prec_{\alpha,\beta}$ and $\succ_{\alpha,\beta}$ are defined respectively by the   \eqref{3-BiH-pre-Lie-twis-cro} and \eqref{twist-struc-BiH-dend}.
\end{thm}
\begin{proof}
By Theorem \ref{3-BiH-pre-Lie-twis-thm} and Proposition \ref{twist-BiH-dend-pro} respectively, $(\mathfrak{g},\{\cdot,\cdot,\cdot\}_{\alpha,\beta},\epsilon,\alpha,\beta)$ is a $3$-BiHom-pre-Lie color algebra and $(\mathfrak{g},\prec_{\alpha,\beta},\succ_{\alpha,\beta},\epsilon,\alpha,\beta)$ is a BiHom-dendriform color algebra.
We must show that the compatibility conditions \eqref{cond-3-BiH-pre-Poiss-color1}-\eqref{cond-3-BiH-pre-Poiss-color3}  are satisfied.
Since $(\mathfrak{g},\{\cdot,\cdot,\cdot\},\prec,\succ,\epsilon)$ is a noncommutative $3$-pre-Poisson color algebra, for $x,y,z,t\in\mathcal{H}(\mathfrak{g})$, we have
\begin{align*}
& \{\alpha\beta^2(x),\alpha\beta^2(y),z\prec_{\alpha,\beta} \alpha(t)\}_{\alpha,\beta}-
\{\beta^2(x),\beta^2(y),z\}_{\alpha,\beta}\prec_{\alpha,\beta} \alpha\beta(t)\\&=
\{\alpha^2\beta^2(x),\alpha^2\beta^2(y),\alpha\beta(z)\prec\alpha\beta^2(t)\}-\{\alpha^2\beta^2(x),\alpha^2\beta^2(y),\alpha\beta(z)\}\prec\alpha\beta^2(t)\\&= \alpha\beta(z)\prec\Big(\{\alpha^2\beta^2(x),\alpha^2\beta^2(y),\alpha\beta^2(t)\}-\epsilon(y,t)\{\alpha^2\beta^2(x),\alpha\beta^2(t),\alpha^2\beta^2(y)\}\\
&\quad +\epsilon(x,y+t)\{\alpha^2\beta^2(y),\alpha\beta^2(t),\alpha^2\beta^2(x)\}\Big)\\&= \beta(z)\prec_{\alpha,\beta}\Big(\{\alpha^2\beta(x),\alpha^2\beta(y),\alpha\beta(t)\}-\epsilon(y,t)\{\alpha^2\beta(x),\alpha\beta(t),\alpha^2\beta(y)\}\\
&\quad +\epsilon(x,y+t)\{\alpha^2\beta(y),\alpha\beta(t),\alpha^2\beta(x)\}\Big)\\&= \beta(z)\prec_{\alpha,\beta}\Big(\{\alpha\beta(x),\alpha\beta(y),\alpha(t)\}_{\alpha,\beta}
-\epsilon(y,t)\{\alpha\beta(x),\beta(t),\alpha^2(y)\}_{\alpha,\beta}\\
&\quad
+\epsilon(x,y+t)\{\alpha\beta(y),\beta(t),\alpha(x)\}_{\alpha,\beta}\Big),
\\
&\{\alpha\beta(x),\alpha^2\beta(y),\alpha(z)\succ t\}_{\alpha,\beta}-\epsilon(x+y,z)\alpha\beta(z)\succ_{\alpha,\beta}\{\alpha(x),\alpha^2(y),t\}_{\alpha,\beta}\\
&=\{\alpha^2\beta(x),\alpha^3\beta(y),\alpha^2\beta(z)\succ \beta^2(t)\}-\epsilon(x+y,z)\alpha^2\beta(z)\succ\{\alpha^2\beta(x),\alpha^3\beta(y),\beta^2(t)\}\\
&=\Big(\{\alpha^2\beta(x),\alpha^3\beta(y),\alpha^2\beta(z)\}-\epsilon(y,z)\{\alpha^2\beta(x),\alpha^2\beta(z),\alpha^3\beta(y)\} \\
&\quad
+\epsilon(x,y+z)\{\alpha^3\beta(y),\alpha^2\beta(z),\alpha^2\beta(x)\}\Big)\succ\beta^2(t)\\&=\Big(\alpha\{\beta(x),\alpha\beta(y),\alpha(z)\}_{\alpha,\beta}-\epsilon(y,z)\alpha\{\beta(x),\beta(z),\alpha^2(y)\}_{\alpha,\beta} \\
&\quad
+\epsilon(x,y+z)\alpha
\{\alpha\beta(y),\beta(z),\alpha(x)\}_{\alpha,\beta}\Big)\succ\beta^2(t)\\
&=\Big(\{\beta(x),\alpha\beta(y),\alpha(z)\}_{\alpha,\beta}-\epsilon(y,z)\{\beta(x),\beta(z),\alpha^2(y)\}_{\alpha,\beta} \\
&\quad +\epsilon(x,y+z)
\{\alpha\beta(y),\beta(z),\alpha(x)\}_{\alpha,\beta}\Big)\succ_{\alpha,\beta}\beta(t).
\end{align*}
Then, the conditions \eqref{cond-3-BiH-pre-Poiss-color1} and \eqref{cond-3-BiH-pre-Poiss-color2} are satisfied. Similarly, the condition \eqref{cond-3-BiH-pre-Poiss-color3} is also satisfied, which ends the proof.
\end{proof}
\begin{pro}
Let $(\mathfrak{g},\{\cdot,\cdot,\cdot\},\prec,\succ,\epsilon,\alpha,\beta)$ be a noncommutative $3$-BiHom-pre-Poisson color algebra. Then, $(\mathfrak{g},[\cdot,\cdot,\cdot]^C,\mu,\epsilon,\alpha,\beta)$ is a noncommutative $3$-BiHom-Poisson color algebra, where $[\cdot,\cdot,\cdot]^C$ is defined by \eqref{eq:3cc} and $\mu$ is defined by Lemma \ref{BiH-ass-from-BiH-den}.

We call $(\mathfrak{g},[\cdot,\cdot,\cdot]^C,\mu,\epsilon,\alpha,\beta)$ the sub-adjacent noncommutative $3$-BiHom-Poisson color algebra of
$(\mathfrak{g},\{\cdot,\cdot,\cdot\},\prec,\succ,\epsilon,\alpha,\beta)$ and denote it by $\mathfrak{g}^C$, and call $(\mathfrak{g},\{\cdot,\cdot,\cdot\},\prec,\succ,\epsilon,\alpha,\beta)$ the compatible noncommutative $3$-BiHom-pre-Poisson color algebra of $\mathfrak{g}^C$.
\end{pro}
\begin{proof}
By Proposition \ref{pro:subadj}, the even trilinear operation $[\cdot,\cdot,\cdot]^C$ gives a $3$-BiHom-Lie color algebra structure on $\mathfrak{g}$ and Lemma \ref{BiH-ass-from-BiH-den} gives that $(\mathfrak{g},\mu,\epsilon,\alpha,\beta)$ is a noncommutative BiHom-associative color algebra.

By definition of $[\cdot,\cdot,\cdot]^C$ and using  \eqref{cond-3-BiH-pre-Poiss-color1}-\eqref{cond-3-BiH-pre-Poiss-color3}, we obtain, for $x,y,z,t\in\mathcal{H}(\mathfrak{g})$,
\begin{align*}
& [\alpha\beta(x),\alpha\beta(y),\mu(z,t)]^C= [\alpha\beta(x),\alpha\beta(y),z\prec t]^C+[\alpha\beta(x),\alpha\beta(y),z\succ t]^C\\
&=\{\alpha\beta(x),\alpha\beta(y),z\prec t\}-\epsilon(y,z+t)\{\alpha\beta(x),\alpha^{-1}\beta(z)\prec\alpha^{-1}\beta(t),\alpha^2(y)\}\\
&\quad +\epsilon(x,y+z+t)\{\alpha\beta(y),\alpha^{-1}\beta(z)\prec\alpha^{-1}\beta(t),\alpha^2(x)\}+\{\alpha\beta(x),\alpha\beta(y),z\succ t\}\\
&\quad -\epsilon(y,z+t)\{\alpha\beta(x),\alpha^{-1}\beta(z)\succ\alpha^{-1}\beta(t),\alpha^2(y)\}\\
&\quad +\epsilon(x,y+z+t)
\{\alpha\beta(y),\alpha^{-1}\beta(z)\succ\alpha^{-1}\beta(t),\alpha^2(x)\}\\
&=\beta(z)\prec\Big(\{\alpha(x),\alpha(y),t\}-\epsilon(y,t)\{\alpha(x),\alpha^{-1}\beta(t),\alpha^2\beta^{-1}(y)\}\\
&\quad +\epsilon(x,y+t)\{\alpha(y),\alpha^{-1}\beta(t),\alpha^2\beta^{-1}(x)\}\Big)+\{\beta(x),\beta(y),z\}\prec\beta(t)\\
&\quad +\Big(\{\beta(x),\beta(y),z\}-\epsilon(y,z)\{\beta(x),\alpha^{-1}\beta(z),\alpha(y)\}+\epsilon(x,y+z)\{\beta(y),\alpha^{-1}\beta(z),\alpha(x)\}\Big)\succ\beta(t)\\
&\quad +\epsilon(x+y,z)\beta(z)\succ\{\alpha(x),\alpha(y),t\}-\epsilon(y,z+t)\epsilon(x,z)\beta(z)\succ\{\alpha(x),\alpha^{-1}\beta(t),\alpha^2\beta^{-1}(y)\}\\
&\quad -\epsilon(y,z)
\{\beta(x),\alpha^{-1}\beta(z),\alpha(y)\}\prec\beta(t)
+\epsilon(x,y+z+t)\epsilon(y,z)\beta(z)\succ\{\alpha(y),\alpha^{-1}\beta(t),\alpha^2\beta^{-1}(x)\}\\
&\quad +\epsilon(x,y+z)\{\beta(y),\alpha^{-1}\beta(z),\alpha(x)\}\prec\beta(t)\\
&=\Big(\{\beta(x),\beta(y),z\}-\epsilon(y,z)\{\beta(x),\alpha^{-1}\beta(z),\alpha(y)\}+\epsilon(x,y+z)\{\beta(y),\alpha^{-1}\beta(z),\alpha(x)\}\Big)\prec\beta(t)\\
&\quad+\Big(\{\beta(x),\beta(y),z\}-\epsilon(y,z)\{\beta(x),\alpha^{-1}\beta(z),\alpha(y)\}+\epsilon(x,y+z)\{\beta(y),\alpha^{-1}\beta(z),\alpha(x)\}\Big)\succ\beta(t)\\
&\quad+\beta(z)\prec\Big(\{\alpha(x),\alpha(y),t\}-\epsilon(y,t)\{\alpha(x),\alpha^{-1}\beta(t),\alpha^2\beta^{-1}(y)\}\\
&\hspace{4cm}+\epsilon(x,y+t)\{\alpha(y),\alpha^{-1}\beta(z),\alpha^2\beta^{-1}(x)\}\Big)\\
&\quad +\beta(z)\succ\Big(\{\alpha(x),\alpha(y),t\}-\epsilon(y,t)\{\alpha(x),\alpha^{-1}\beta(t),\alpha^2\beta^{-1}(y)\}\\
&\hspace{4cm} +\epsilon(x,y+t)\{\alpha(x),\alpha^{-1}\beta(z),\alpha^2\beta^{-1}(x)\}\Big)\\
&=\mu\big([\beta(x),\beta(y),z]^C,\beta(t)\big)+\epsilon(x+y,z)\mu\big(\beta(z),[\alpha(x),\alpha(y),t]^C\big).
\end{align*}
Then, the BiHom Leibniz rule \eqref{BiHom-Leibniz-role} is satisfied by $[\cdot,\cdot,\cdot]^C$ and $\mu$ wich gives the result.
\end{proof}
\begin{thm}
Let $(\mathfrak{g},\{\cdot,\cdot,\cdot\},\mu,\epsilon,\alpha,\beta)$ be a noncommutative $3$-BiHom-Poisson color algebra and $T:V\to\mathfrak{g}$ be a Kupershmidt operator on $\mathfrak{g}$ with respect to representation $(V,\rho,\mathfrak l,\mathfrak r,\alpha_V,\beta_V)$. Then, the operations
$\{\cdot,\cdot,\cdot\}_V,\;\prec_V$ and $\succ_V$ define a noncommutative $3$-BiHom-pre-Poisson color algebra structure on $V$, where $\{\cdot,\cdot,\cdot\}_V$ is given by Theorem \ref{3-BiH-pre-Lie-idu-Kup} and $\prec_V,\;\succ_V$ are given by Theorem \ref{BiH-dedr-color-ind-Kup}.
\end{thm}
\begin{proof}
By using Theorems \ref{3-BiH-pre-Lie-idu-Kup} and \ref{BiH-dedr-color-ind-Kup}, we deduce that $(V,\{\cdot,\cdot,\cdot\}_V,\epsilon,\alpha_V,\beta_V)$ is a $3$-BiHom-pre-Lie color algebra and $(V,\prec_V,\succ_V,\epsilon,\alpha_V,\beta_V)$ is a BiHom-dendriform color algebra. It remains to show that conditions \eqref{cond-3-BiH-pre-Poiss-color1}-\eqref{cond-3-BiH-pre-Poiss-color3} are satisfied.
For any $x,y,z,t\in\mathcal{H}(V)$, by using the definitions of
$\{\cdot,\cdot,\cdot\}_V$ and $\prec_V$ and the fact that $T$ is Kupershmidt of the $3$-BiHom-Poisson color algebra $\mathfrak{g}$, we obtain
\begin{align*}
&\{\alpha_V\beta_V(x),\alpha_V\beta_V(y),z\prec_Vt\}_V-
\{\beta_V(x),\beta_V(y),z\}_V\prec_V\beta_V(t) \\
&=\epsilon(z,t)\{\alpha_V\beta_V(x),\alpha_V\beta_V(y),\mathfrak{r}(T(t))z\}_V-\epsilon(t,x+y+z)\mathfrak{r}(T(\beta_V(t)))\{\beta_V(x),\beta_V(y),z\}_V\\
&= \epsilon(z,t)\rho(\alpha\beta(T(x)),\alpha\beta(T(y)))\mathfrak{r}(T(t))z- \epsilon(t,x+y+z)\mathfrak{r}(\beta(T(t)))\rho(\beta(T(x)),\beta(T(y)))z\\
&=\epsilon(t,x+y+z)\mathfrak{r}(\beta(T(t)))\rho(\beta(T(x)),\beta(T(y)))z\\
&\quad -\epsilon(t,x+y+z)\mathfrak{r}(\beta(T(t)))\rho(\beta(T(x)),\beta(T(y)))z+\epsilon(z,t)\mathfrak{r}(\{\alpha(T(x)),\alpha(T(y)),T(t)\})\beta_V(z)\\
&= \epsilon(z,t)\mathfrak{r}(\{T(\alpha_V(x)),T(\alpha_V(y)),T(t)\})\beta_V(z)\\
&=\epsilon(z,t)\mathfrak{r}\Big(T\Big(\rho(T(\alpha_V(x)),T(\alpha_V(y)))t-\epsilon(y,t)\rho(T(\alpha_V(x)),T(\alpha_V^{-1}\beta_V(t)))\alpha_V^2\beta_V^{-1}(y)\\
&\quad +\epsilon(x,y+t)\rho(T(\alpha_V(y)),T(\alpha_V^{-1}\beta_V(t)))\alpha_V^2\beta_V^{-1}(x)\Big)\Big)\beta_V(z)\\
&=\epsilon(x+y,z)\beta_V(z)\prec_V\rho(T(\alpha_V(x)),T(\alpha_V(y)))t\\
&\quad -\epsilon(x+y,z)\epsilon(y,t)\beta_V(z)\prec_V\rho(T(\alpha_V(x)),T(\alpha_V^{-1}\beta_V(t)))\alpha_V^2\beta_V^{-1}(y)\\
&\quad +\epsilon(x,y+z+t)\epsilon(x,t)\epsilon(z,t)\beta_V(z)\prec_V\rho(T(\alpha_V(y)),T(\alpha_V^{-1}\beta_V(t)))\alpha_V^2\beta_V^{-1}(x)\\
&=\epsilon(z,x+y)\beta_V(z)\prec_V\Big(\{\alpha_V(x),\alpha_V(y),t\}_V-\epsilon(y,t)\{\alpha_V(x),\alpha_V^{-1}\beta_V(t),\alpha^2\beta_V^{-1}(y)\}_V\\
&\quad+\epsilon(x,y+t)\{\alpha_V(y),\alpha_V^{-1}\beta_V(t),\alpha^2\beta_V^{-1}(x)\}_V\Big),
\end{align*}
then, the condition \eqref{cond-3-BiH-pre-Poiss-color1} is satisfied. We have also
\begin{align*}
& \{\alpha_V\beta_V(x),\alpha_V^2\beta_V(y),\alpha_V(z)\succ_V t\}_V-\epsilon(x+y,z)\alpha_V\beta_V(z)\succ_V\{\alpha_V(x),\alpha_V^2(y),t\}_V\\
&= \{\alpha_V\beta_V(x),\alpha_V^2\beta_V(y),\mathfrak{l}(T\alpha_V(z)) t\}_V
\epsilon(x+y,z)\mathfrak{l}(T\alpha_V\beta_V(z))\{\alpha_V(x),\alpha_V^2(y),t\}_V\\
&=\rho\big(T(\alpha_V\beta_V(x)),T(\alpha_V^2\beta_V(y))\big)\mathfrak{l}(T(\alpha_V(z)))t \\
&\quad -\epsilon(x+y,z)\mathfrak{l}(T(\alpha_V\beta_V(z)))\rho\big(T(\alpha_V(x)),T(\alpha_V^2(y))\big)t\\
&=\mathfrak{l}\big(\{T(\beta_V(x)),T(\alpha_V\beta_V(y)),T(\alpha_V(z))\}\big)\beta_V(t)
\\
&\quad+\epsilon(x+y,z)\mathfrak{l}(T(\alpha_V\beta_V(z)))\rho\big(T(\alpha_V(x)),T(\alpha_V^2(y))\big)t\\
&\quad-\epsilon(x+y,z)\mathfrak{l}(T(\alpha_V\beta_V(z)))\rho\big(T(\alpha_V(x)),T(\alpha_V^2(y))\big)t\\
&= \mathfrak{l}\big(\{T(\beta_V(x)),T(\alpha_V\beta_V(y)),T(\alpha_V(z))\}\big)\beta_V(t)\\
&=\mathfrak{l}\Big(T\big(\rho(T(\beta_V(x)),T(\alpha_V\beta_V(y)))\alpha_V(z)-\epsilon(y,z)\rho(T(\beta_V(x)),T(\beta_V(z)))\alpha_V^2(y)\\
&\quad +\epsilon(x,y+z)\rho(T(\alpha_V\beta_V(y)),T(\beta_V(z)))\alpha_V(x)\big)\Big)\beta_V(t)\\&=\Big(\rho(T(\beta_V(x)),T(\alpha_V\beta_V(y)))\alpha_V(z)-\epsilon(y,z)\rho(T(\beta_V(x)),T(\beta_V(z)))\alpha_V^2(y)\\
&\quad +\epsilon(x,y+z)\rho(T(\alpha_V\beta_V(y)),T(\beta_V(z)))\alpha_V(x)\Big)\succ_V\beta_V(t)\\&=\Big(\{\beta_V(x),\alpha_V\beta_V(y),\alpha_V(z)\}_V-\epsilon(y,z)\{\beta_V(x),\beta_V(z),\alpha_V^2(y)\} \\
&\quad +\epsilon(x,y+z)\{\alpha_V\beta_V(y),\beta_V(z),\alpha_V(x)\}\Big)\succ_V\beta_V(t),
\end{align*}
which shows that the condition \eqref{cond-3-BiH-pre-Poiss-color2} is satisfied. The same direct computation gives the condition \eqref{cond-3-BiH-pre-Poiss-color3}. The Theorem is proved.
\end{proof}
\begin{cor}
Let $(\mathfrak{g},\{\cdot,\cdot,\cdot\},\mu,\epsilon,\alpha,\beta)$ be a noncommutative $3$-BiHom-Poisson color algebra and $\mathcal{R}:\mathfrak{g}\to\mathfrak{g}$ be a Rota-Baxter operator of weight zero on $\mathfrak{g}$ with respect to a representation $(V,\rho,l,r,\alpha_V,\beta_V)$. Then, the tuple $(\mathfrak{g},\{\cdot,\cdot,\cdot\}_{\mathcal{R}},\prec_\mathcal{R},\succ_{\mathcal{R}},\epsilon,\alpha,\beta)$ is a $3$-BiHom-pre-Poisson color algebra, where $\{\cdot,\cdot,\cdot\}_{\mathcal{R}}$ is given in Corollary \ref{ind-RB-3-BiH-pre-Lie-color} and $\prec_{\mathcal{R}},\;\succ_{\mathcal{R}}$ are given in Corollary \ref{ind-RB-BiH-dend-color}.
\end{cor}

\section*{Acknowledgment}

The authors thank Professor Sami Mabrouk for his helpful discussions and suggestions.


\end{document}